\documentclass{article}
\usepackage{graphicx}
\usepackage[english]{babel}
\usepackage{caption}
\usepackage{subcaption}
\usepackage{amsmath, amsfonts, amssymb, amsthm, bm}  
\usepackage{mathrsfs}
\usepackage{url}
\usepackage{hyperref} 

\theoremstyle{plain}
\newtheorem{theorem}{Theorem}[section]
\newtheorem{proposition}[theorem]{Proposition}
\newtheorem{lemma}[theorem]{Lemma}
\newtheorem{corollary}[theorem]{Corollary}
\newtheorem{remark}[theorem]{Remark}

\newtheorem{conjecture}[theorem]{Conjecture}

\theoremstyle{definition}

\allowdisplaybreaks

\newcommand{\bydef}{\stackrel{\textnormal{\tiny def}}{=}}
\newcommand{\N}{\mathbb{N}}
\newcommand{\Z}{\mathbb{Z}}
\newcommand{\R}{\mathbb{R}}
\newcommand{\C}{\mathbb{C}}
\newcommand{\cX}{\mathcal{X}}

\usepackage{color}

\begin{document}
\title{
From the Lagrange Triangle to the Figure Eight Choreography: Proof of Marchal's Conjecture
}
\author{
Renato Calleja
\thanks
{IIMAS, Universidad Nacional Aut\'{o}noma de M\'{e}xico, Apdo. Postal 20-726, C.P. 01000, M\'{e}xico D.F., M\'{e}xico. {\tt calleja@mym.iimas.unam.mx}.}
\and Carlos Garc\'{i}a-Azpeitia
\thanks
{IIMAS, Universidad Nacional Aut\'{o}noma de M\'{e}xico, Apdo. Postal 20-726, C.P. 01000, M\'{e}xico D.F., M\'{e}xico. {\tt cgazpe@ciencias.unam.mx}.}
\and Olivier H\'{e}not
\thanks
{National Taiwan University, Department of Mathematics, No. 1 Sec. 4 Roosevelt Rd., 10617 Taipei, Taiwan. {\tt olivierhenot@ntu.edu.tw}.}
\and Jean-Philippe Lessard
\thanks
{McGill University, Department of Mathematics and Statistics, 805 Sherbrooke Street West, Montreal, QC, H3A 0B9, Canada. {\tt jp.lessard@mcgill.ca}.}
\and Jason D. Mireles James
\thanks
{J.D.M.J was partially supported by NSF grant DMS 2307987.}
\thanks
{Florida Atlantic University, Department of Mathematical Sciences, Science Building, Room 234, 777 Glades Road, Boca Raton, Florida, 33431, USA. {\tt jmirelesjames@fau.edu}.}
}

\date{}

\maketitle

%\tableofcontents

\begin{abstract}
For the three body problem with equal masses, we prove that the most symmetric continuation class of Lagrange's equilateral triangle solution, also referred to as the $P_{12}$ family of Marchal \cite{MR1124619}, contains the remarkable figure eight choreography discovered by Moore in 1993 \cite{DiscoveryEight}, and proven to exist by Chenciner and Montgomery in 2000 \cite{ProofEight}.
This settles a conjecture of Marchal which dates back to the 1999 conference on Celestial Mechanics in Evanston Illinois, celebrating Donald Saari's 60th birthday \cite{Fejoz}.
\end{abstract}

\begin{center}
{\bf \small Key words:}
{\small three body problem, continuation class of the Lagrange triangle,\\ 
figure eight choreography, computer-assisted proof, Marchal's conjecture}
\end{center}

%%%%%%%%%%%%%%%%%%%%%%%%%%

%!TEX root = marchal.tex

\section{Introduction}

In a paper published in 1993, Moore describes a periodic solution of the gravitational $N$-body problem where three bodies of equal mass follow one another around a closed eight-shaped curve in the plane \cite{DiscoveryEight}.
The orbit was discovered by numerically minimizing the Newtonian potential over the space of closed plane curves via gradient descent.
This \emph{remarkable solution} of the $N$-body problem was discovered independently by Chenciner and Montgomery, whose variational proof of its existence was published in the year 2000 \cite{ProofEight}.
See the right frame of Figure \ref{fig:triangle_eight_1} in the present work, and also \cite{scholarpedia} for graphical illustrations of the three body eight.
These days such an orbit is called a choreography, as the bodies appear to dance around a prescribed curve in space.
The paper \cite{MR1919833} by Chenciner, Gerver, Montgomery and Sim\'{o} provides a lovely introduction to the subject of choreographies in particle systems.

Shortly after the announcement of the eight by Chenciner and Montgomery, Marchal published a paper investigating the discrete symmetries of $3$-body choreographies \cite{P12family}.
In particular, recalling that non-circular periodic solutions of the equal mass three body problem have at most 12 space-time symmetries, he studied the properties of the most symmetric family of \emph{relative periodic orbits} bifurcating from Lagrange's equilateral triangle by continuation with respect to the frequency of a rotating frame.
By a \emph{relative periodic orbit}, we mean a solution which is periodic after changing to an appropriate co-rotating frame.
Marchal referred to this most symmetric continuation class as the $P_{12}$ family, and in the same reference showed that:

\begin{itemize}
\item \textbf{The $\bm{P_{12}}$ family is not empty:} there exists a family of relative periodic orbits bifurcating from the Lagrange triangle and having the maximal allowed 12 symmetries.

\item \textbf{The $\bm{P_{12}}$ family comprises relative choreographies:} in a suitable rotating coordinate system, each periodic orbit 
in the $P_{12}$ family can be viewed as a single closed curve in $\mathbb{R}^3$.
The three bodies follow one another around this curve, spaced out by a time shift of $1/3$ the period.

\item \textbf{Variational properties:} the $P_{12}$ family minimizes the action between appropriate terminal conditions.
\end{itemize} 

Next he observed that the three body figure eight evinces all 12 symmetries of the $P_{12}$ family.
Inspired by this, he proposed that $P_{12}$ actually contains the eight.
We refer to this as \textit{Marchal's conjecture}.

The local structure of the $P_{12}$ family, and hence the plausibility of Marchal's conjecture, is further considered in the papers \cite{MR2425050} and also \cite{MR2134901}, of Chenciner and F\'{e}joz, and of Chenciner, F\'{e}joz and Montgomery respectively.
In particular, the former paper applies a center manifold analysis in a frame rotating with the Lagrange triangle and shows that, near the triangle, the $P_{12}$ family is diffeomorphic to a cylinder, and that the cylinder is transverse to the plane of the Lagrange triangle (in fact it is a vertical Lyapunov family attached Lagrange's relative equilibrium).
So, in appropriate rotating coordinates, $P_{12}$ is described as a smoothly varying one-parameter family of out-of-plane choreographic orbits, which bifurcate from the Lagrangian triangle.

In the latter of the two references just cited, the authors make a local analysis of the eight.
They show -- via variational techniques -- that there are three families of choreographic orbits bifurcating from the figure eight, and that one of these has the symmetries of the $P_{12}$ family.
They also provide numerical evidence which suggests that these are the only families of relative choreographies bifurcating from the figure eight.
Together, these two papers show that there are no local obstructions to Marchal's conjecture.

It remains to demonstrate that the $P_{12}$ family actually bridges the gap between the work of \cite{MR2425050} and \cite{MR2134901}.
The contribution of the present work is to complete this picture, settling Marchal's conjecture in the affirmative.
A precise statement of our main result is given in Theorem \ref{thm:main}.
We summarize it informally as follows.

\begin{theorem}\label{thm:MarchalWasRight}
In the three body problem with equal masses, the figure eight choreography of Moore, Chenciner, and Montgomery is contained in the $P_{12}$ family of Marchal.
\end{theorem}

This theorem, combined with Proposition 3 in \cite{MR4208440}, implies the existence of a dense subset of $P_{12}$, corresponding to choreographies in the inertial frame with arbitrarily large period and lying on the surface of a topological cylinder.

The proof of Theorem \ref{thm:MarchalWasRight} is constructive, and proceeds by several interconnected steps.

\begin{enumerate}
\item \textbf{Symmetric function space reduction:} we reformulate the problem in a rotating
frame and impose (in Fourier space) the symmetries of Marchal's $P_{12}$ family.
This, in particular, removes all three body bifurcations which do not result in relative choreographies.
The result is a system of algebraic advanced/delay differential equations with discrete symmetries.
Our task is to establish the existence of a branch of solutions of the symmetrized problem, parameterized by the frequency of the rotating frame, starting from the Lagrange triangle and extending to the figure eight.

\item \textbf{Desingularization at the endpoints:} the Lagrange triangle and the figure eight are both planar solutions of the three body problem, and the desired $P_{12}$ branch is transverse to each of these planes (see the right frame of Figure \ref{fig:triangle_eight_1}).
Despite the discrete symmetries mentioned above, and because of the planar symmetry, neither the triangle nor the eight are isolated as choreographies.
Put another way, the $P_{12}$ branch begins and ends at a symmetry breaking bifurcation, and because of this it is necessary to 
desingularize the problem at each endpoint.
We introduce an ``unfolding parameter'' which desingularizes this symmetry breaking along the entire branch.

\item \textbf{A posteriori argument:} we numerically compute a candidate branch of solutions to the desingularized symmetrized problem (see the left frame of Figure \ref{fig:triangle_eight_1}), and apply a uniform contraction argument in a neighborhood of this candidate.
In this way, we prove the existence of a true branch of solutions near (in an appropriate norm) the candidate solution branch.
The procedure is carried out in a Banach space of rapidly decaying Fourier-Chebyshev series coefficients, where we are able to express the entire solution branch using a single Chebyshev expansion in the frequency parameter.

\item \textbf{Analytical bound derivation:} deriving the analytical bounds needed for the uniform contraction argument constitutes a significant portion of the paper.
The assumptions of the uniform contraction theorem are distilled down to a list of inequalities, involving the Fourier-Chebyshev
coefficients of the candidate branch.
Ultimately, verifying these hypotheses demands a very large, but finite number of calculations.

\item \textbf{Interval arithmetic:} given the significant computational complexity of the inequalities, and the large number of Fourier-Chebyshev coefficients representing our candidate, manual calculations are not feasible. Instead, we use a computer for this task, employing an interval arithmetic library \cite{IntervalArithmeticJulia} as well as the \texttt{RadiiPolynomial} library \cite{Henot2021-2} implemented in the Julia programming language \cite{Julia}.
Readers interested in independently verifying the necessary inequalities at the final step of the proof can use the following Jupyter Notebook \cite{Henot2024}:

 \url{https://github.com/OlivierHnt/MarchalConjecture.jl}

The necessary checks are completed in less than an hour on a laptop with a M1 chip and 8GB of RAM.
\end{enumerate}

We also recall the remarks of Marchal at the end of Section 4 in \cite{P12family} (and the closing remarks of Section 4.2 in \cite{craig2008hamiltonian}), where it is explained that if the Lagrange triangle can be continued to the figure eight through the $P_{12}$ family, then -- by further symmetry considerations -- the family continues 
through the figure eight, and returns to the Lagrange triangle with opposite orientation.
Following the branch further leads back to the figure eight (with reversed orientation) and finally back to the Lagrange triangle.
So, the branch proved to exist in the present work is precisely a quarter of the $P_{12}$ family and, as a corollary to Theorem \ref{thm:MarchalWasRight}, we obtain the global description of $P_{12}$.

\begin{figure}[!ht]
\centering
\begin{subfigure}[b]{0.43\textwidth}
\centering
\includegraphics[width=\textwidth]{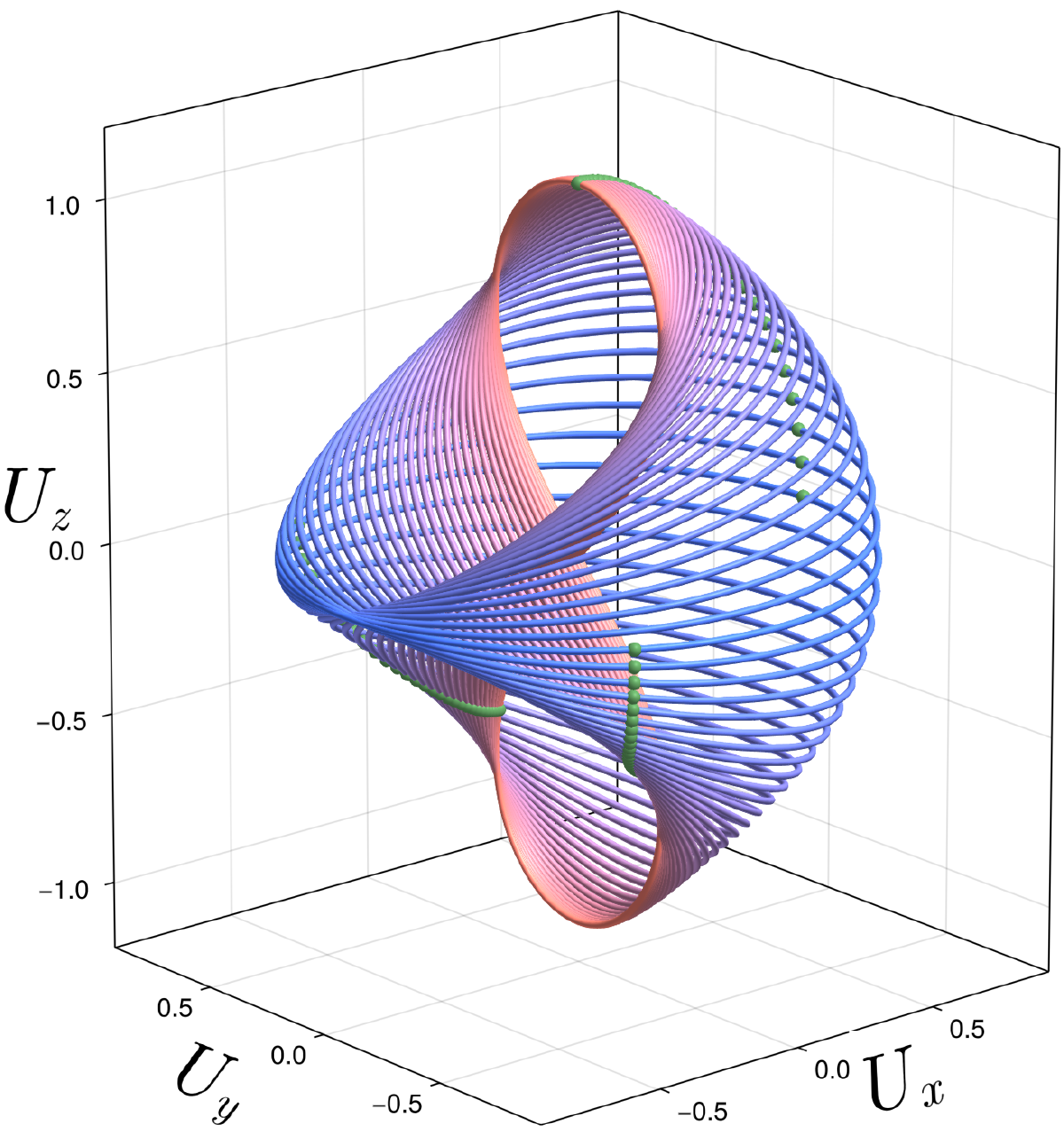}
\label{fig:P12_family}
\caption{Thirty-three orbits sampled from the continuous branch
of solutions proven to exist by Theorem \ref{thm:MarchalWasRight}.
Each curve is a relative periodic choreography of the three body problem 
with equal masses.}
\end{subfigure}
\hspace{0.1\textwidth}
\begin{subfigure}[b]{0.45\textwidth}
\centering
\includegraphics[width=\textwidth]{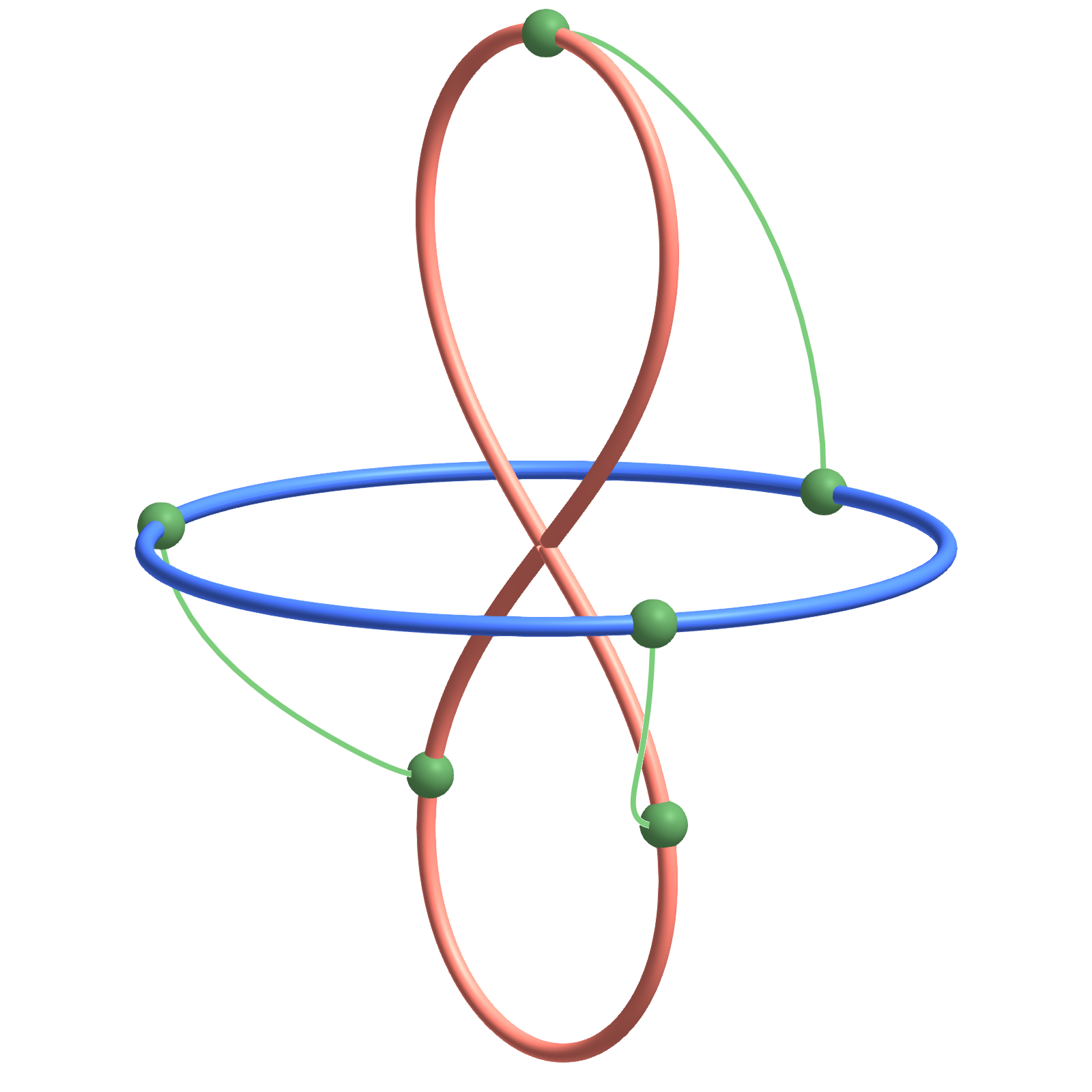}
\caption{The Lagrange triangle (blue) and the figure eight (red).
Each of these choreographies is planar, and the $P_{12}$ family -- comprised of relative choreographies -- is transverse to both planes.}
\end{subfigure}
\caption{Branch of relative choreographies connecting the Lagrange 
triangle to the figure eight. A dense set of the relative choreographies in the family corresponds to choreographies in inertial frame.}
\label{fig:triangle_eight_1}
\end{figure}

%\begin{figure}[!ht]
%\centering
%\begin{subfigure}[b]{0.4\textwidth}
%\centering
%\includegraphics[width=\textwidth]{figures/triangle.png}
%\caption{}
%\end{subfigure}
%\hspace{0.1\textwidth}
%\begin{subfigure}[b]{0.4\textwidth}
%\centering
%\includegraphics[width=\textwidth]{figures/eight.png}
%\caption{}
%\end{subfigure}
%\caption{}
%\label{fig:triangle_eight_1}
%\end{figure}

The remainder of the paper is structured as follows.
Section \ref{sec:setUp} covers the symmetry reduction, the desingularization at the figure eight and the Lagrange triangle, and the treatment of non-polynomial nonlinearities.
These concepts collectively reformulate the problem as a functional equation.
In Section \ref{sec:proof}, we define the norms and establish the necessary bounds to prove the existence of a branch of solutions of the functional equation using the contraction mapping theorem.
Section \ref{sec:aposteriori} completes the demonstration of Marchal’s conjecture by establishing that the branch of zeros identified in Section \ref{sec:proof} corresponds to the $P_{12}$ family of Marchal and ends at the figure eight of Chenciner and Montgomery.
Section \ref{sec:future} concludes the paper by outlining potential future work.

\subsection{Some historical remarks}

The first choreographic solution of the $N$-body problem appeared in Lagrange's 1772 publication of the equilateral triangle solution of the $3$-body problem \cite{lagrangeTriangle}.
For appropriate initial velocities, this special solution involves three bodies of equal mass moving in circular motion at constant speed.
The result was generalized more than a century later by Hoppe \cite{nGonPaper}, who showed that for every $N$ many bodies there is a circular periodic orbit where the $N$ equal masses are located at the vertices of a rotating regular $N$-gon.
Today we would refer to all of these polygonal solutions as trivial choreographies.
It is a testament to the richness of the gravitational $N$-body problem that it took more than 200 years for a non-trivial choreography to be discovered by Moore \cite{DiscoveryEight}.

The story of Marchal's conjecture picks up in the Fall of 1999 when Chenciner, Gerver, Montgomery and Sim\'{o} were all congregated at the Paris Observatory's Institut de M\'{e}canique C\'{e}leste et de Calcul des \'{E}ph\'{e}m\'{e}rides (IMCCE).
Montgomery had recently published the paper \cite{MR1610784}, where he used variational arguments to prove the existence of many braided periodic solutions of the $N$-body problem with the strong $1/r^a$ potential with $a=2$.
He further conjectured that many of these braids should continue to the Newtonian potential taking $a \to 1$.

Chenciner and Montgomery worked on these and similar problems in the winter of 1999 and, as recalled in the introduction of \cite{MR1919833}, by December they had proven the existence of the three body eight for the Newtonian potential.
This sets off -- in the words of the same introduction -- a flurry of work, and soon Gerver and Sim\'{o} had computed a whole zoo of new $N$-body choreographies for $N$ as large as $799$.
It was during this period that Sim\'{o} actually coined the term \emph{choreography}.

Later the same month, at a conference on Celestial Mechanics in Evanston, Illinois, honoring the 60th birthday of Donald Saari, Chenciner gave a presentation announcing the proof of the figure eight.
According to the first paragraph of \cite{P12family}, Marchal was in attendance.
Indeed, in his 2010 M\'{e}moire d'Habilitation \cite{Fejoz}, F\'{e}joz recalls that upon hearing the talk, Marchal immediately realized that the figure eight could be related to the Lagrange triangle through his $P_{12}$ family, and Marchal's conjecture was born.
The conjecture is referred to as ``Partially Numerical Theorem $19$'' in \cite{Fejoz}, and its status is reviewed in detail in Section 3.4 of the same document.
See also the 2009 paper \cite{MR2480953} of Chenciner and F\'{e}joz for a detailed numerical evidence supporting the conjecture.

After Evanston, interest in choreographic solutions spread rapidly.
In 2000 the paper of Chenciner and Montgomery, describing the proof of the figure eight, appeared in the Annals of Mathematics \cite{ProofEight}.
Marchal's paper \cite{P12family}, about the symmetries of the figure eight, appeared in Celestial Mechanics and Dynamical Astronomy the same year.
This marked the first appearance of the conjecture in print, and from this point on, the mathematical theory progressed as discussed in the previous section.

During the next two years, the papers \cite{MR1905315,MR1884902} by Sim\'{o} provided detailed numerical accounts of many additional symmetric and asymmetric families of $N$-body choreographies, and reported the surprising numerical result that the three body figure eight choreography appeared to be stable in the Hamiltonian sense.

We also mention that, as recalled in both the references of \cite{ProofEight} and the acknowledgements of \cite{MR1919833}, it was Phil Holmes and Robert MacKay who brought the paper \cite{DiscoveryEight} of Moore to the attention of Chenciner and Montgomery.
This apparently happened during the galley proofing of \cite{ProofEight}.
Chenciner and Montgomery in turn seem to have brought the paper of Moore to the attention of a much larger audience than before.
The paper seems to have been cited seven times between 1993 and 2000, and has more than three hundred citations by the date of the present manuscript.

It should also be noted that the first computer-assisted existence theorems for choreographies appeared fairly soon after the Evanston conference.
The paper \cite{MR2012847} of Kapela and Zgliczy\'{n}ski, appearing in Nonlinearity in 2003, reproved the existence of the figure eight and also established its convexity (a difficult property to obtain using variational methods).
In the same reference they also prove the existence of some symmetric planar 4, 6, and 8 body choreographies which had appeared earlier in the numerical studies of Sim\'{o} \cite{MR1905315,MR1884902}.
Later in \cite{MR2312391}, Kapela and Sim\'{o} proved the linear stability of the figure eight, and the existence of many additional non-symmetric choreographies from \cite{MR1905315,MR1884902}.
In \cite{MR3622273}, the same authors prove the KAM stability of the figure eight in an appropriate energy/symmetry submanifold.
Along similar lines, the dissertation \cite{minton2013computer} presents a computer-assisted framework for proving the existence of planar choreographies; a charming feature of this work is an interactive interface that allows the user to draw a candidate loop, and, upon convergence of Newton’s method from the sketch, tries to establish a rigorous existence proof.

Furthermore, we mention the work \cite{MR4208440} by four of the authors of the present manuscript.
This paper provides computer-assisted existence proofs for spatial torus knot choreographies (including a five body trefoil knot), and laid the foundation for the current work.
More precisely, the computer-assisted proofs in the reference just cited employ a functional analytic framework based on advanced/delay equations similar to the one discussed in Section \ref{sec:setUp}.
We note however that the results of \cite{MR4208440} are localized in parameter space, while the results of the present describe a global branch containing the triangle and the eight.

Lastly, we mention the papers of Arioli, Barutello, and Terracini \cite{Arioli2006,Barutello_2004}, where the authors study choreographic solutions of the three body problem using computer-assisted methods similar to those mentioned above.
Specifically, the work in \cite{Arioli2006} begins with a variational proof of a choreographic mountain pass solution between the Lagrangian minimizers (circular choreographies) of period $2\pi$ and $4\pi$ respectively.
Next, they use a numerical bisection method to approximately locate this mountain pass bifurcation.
Finally, they employ the validated continuation scheme developed by Arioli and Koch in \cite{MR2679365} to establish the existence of a continuous branch of mountain pass solutions near the numerically detected bifurcation point.
Here, it is essential to mention that, at the end of their introduction, the authors of \cite{Arioli2006} themselves explain that their family cannot be the $P_{12}$ family of Marchal.

The preceding discussion is by no means a complete review of the literature on choreographic orbits in celestial mechanics.
An excellent overview by Montgomery, with many additional references, is found at \cite{scholarpedia}.
Similarly, computer-assisted methods based on a posteriori validation have a long history going back to the works of Lanford, Eckmann, Koch, and Wittwer in the mid-1980's on the Fegenbaum conjectures \cite{MR0648529,MR0883539,MR0727816}.
The use of such computer-assisted methods in nonlinear analysis and dynamical systems theory is by now a thriving subindustry, and to attempt even a terse overview would be a task beyond the scope of the present paper.
We refer the interested reader to the review articles of \cite{MR0759197,MR1420838,MR3444942,MR3990999} and to the books of \cite{MR2807595,MR3822720,MR3971222} for thorough discussion of this literature.

%!TEX root = marchal.tex

\section{Reformulation of the problem as a functional equation}
\label{sec:setUp}

We consider the motion of three bodies interacting under Newton's gravitational law, depending on their mass, initial position and initial velocity.
The positions $q_1, q_2, q_3$ are solutions of the equations of motion
\begin{equation}\label{eq:N_body}
m_j \ddot{q}_j = - \partial_{q_j} U,
\qquad U \bydef  - \sum_{j < l} \frac{G m_j m_l}{\| q_j - q_l \|},
\qquad j = 1, 2, 3,
\end{equation}
where $(q_1, q_2, q_3) \in\{ (\mathbb{R}^3)^3 \, : \, q_1 \neq q_2, \, q_1 \neq q_3, \, q_2 \neq q_3 \}$ and $G$ denotes the gravitational constant.

Changing the units of length, mass and time by the factors $A, B, C$ respectively allows to scale $G = 1$ under the condition $A^3 = B C^2$.
We study the case of equal masses, and, by altering the unit of mass, we scale $m_1 = m_2 = m_3 = 1$, leaving out one more degree of freedom.
Then, the motion of the bodies is described by the system
\begin{equation}\label{eq:N_body_normalized}
\ddot{q}_j = - \sum_{j \neq l} \frac{q_j - q_l}{\| q_j - q_l \|^3}, \qquad j = 1, 2, 3.
\end{equation}

We set the inertial frame by placing the origin at the center of mass of the Lagrange triangle, whose plane is associated with the horizontal $xy$-plane.
Due to our scaling of mass, the Lagrange triangle constitutes a relative equilibrium of \eqref{eq:N_body_normalized}, rotating horizontally with frequency $1$, whose vertices are located at
\begin{equation}\label{eq:triangle}
q_j (t)
= 3^{-1/6} e^{(t+2\pi j/3)\bar{J}} e_1
= 3^{-1/6}
\begin{pmatrix}
\cos(t+2\pi j/3) \\ \sin(t+2\pi j/3) \\ 0
\end{pmatrix},
\qquad j = 1, 2, 3,
\end{equation}
where $e_1 \bydef (1, 0, 0)$ and the matrix $\bar{J}$ is defined below in \eqref{eq:I_J_matrices}. 
This configuration forms a circular choreography such that $q_j (t)=q_3 (t+2\pi j/3)$ for $j = 1, 2$.
The $P_{12}$ family studied by Marchal is a family of solutions of \eqref{eq:N_body_normalized} in a horizontally rotating frame, parameterized by the rotation frequency, see \cite{P12family} for details.
For this purpose, we introduce the matrices
\begin{equation}\label{eq:I_J_matrices}
\bar{I} \bydef
\begin{pmatrix}
1 & 0 & 0\\
0 & 1 & 0\\
0 & 0 & 0
\end{pmatrix}, \qquad
\bar{J} \bydef
\begin{pmatrix}
0 & -1 & 0\\
1 & 0 & 0\\
0 & 0 & 0
\end{pmatrix}.
\end{equation}
Then the positions of the bodies in Marchal's rotating frame coordinates are given by
\begin{equation}\label{eq:rotating_coordinates}
U_j (t,\Omega) = e^{\Omega t\bar{J}} q_j (t).
\end{equation}
Under a scaling of time (using our last degree of freedom), we consider $U_1, U_2, U_3$ to have period $2\pi$.
Note that in the inertial frame, the solution $q_1, q_2, q_3$ may be quasiperiodic, depending on $\Omega$.
Proposition~3 in \cite{MR4208440} implies that there are infinitely many choreographies, in the inertial frame, lying on the surface of a topological cylinder.

\subsection{Symmetry reduction}

The $P_{12}$ family consists of a family of a specific class of periodic solutions with the \emph{choreography symmetry}
\begin{equation}\label{eq:layout_symmetry}
U_j (t,\Omega) = U_3 (t+4\pi j/3,\Omega),
\qquad j = 1, 2.
\end{equation}
To avoid carrying out an additional subscript, we let $U = U_3$.
Given that the masses of the three bodies are equal, in the rotating frame \eqref{eq:rotating_coordinates} the equations are equivariant under permutations and time-shifts.
Therefore, a solution satisfying \eqref{eq:rotating_coordinates} and \eqref{eq:layout_symmetry} is represented as a solution to the system of delay differential equations
\begin{equation}\label{eq:N_body_delay}
0 = G(U) \bydef
\ddot{U} - 2 \Omega \bar{J} \dot{U} - \Omega^2 \bar{I} U +
\sum_{j=1}^2 \frac{U-S^jU}{\| U - S^j U \|^3},
\end{equation}
where
\begin{equation}\label{eq:shift_definition}
[S^jU](t,\Omega) \bydef U(t + 4\pi j/3, \Omega).
\end{equation}
Notice that in these coordinates, the triangular relative equilibrium \eqref{eq:triangle} satisfies the symmetries \eqref{eq:layout_symmetry} and is represented by
\[
U(t, 1) = 3^{-1/6} e^{2t\bar{J}}e_{1}.
\]

For the sake of simplicity, in this section we study formally the operator $G : Y \to Y$ in the space
\begin{equation}
Y \bydef C^\infty(\mathbb{R}/2\pi\mathbb{Z} \times [0,1], \mathbb{R}^3),
\end{equation}
and conscientiously do not prescribe a norm.
In Section \ref{sec:proof}, we will use a Banach space of analytic functions contained in $Y$.
The operator $G$ is equivariant under the action of the group $(\sigma,\tau) \in \mathbb{Z}_2 \times \mathbb{Z}_2$ in $Y$ given by
\begin{align}
\sigma \cdot U(t, \Omega) &\bydef R_y U(-t, \Omega),\\
\tau \cdot U(t, \Omega) &\bydef R_z U(t+\pi, \Omega),
\end{align}
where
\begin{equation}
R_y \bydef
\begin{pmatrix}
1 & 0 & 0\\
0 & -1 & 0\\
0 & 0 & 1
\end{pmatrix}, \qquad
R_z \bydef
\begin{pmatrix}
1 & 0 & 0\\
0 & 1 & 0\\
0 & 0 & -1
\end{pmatrix}.
\end{equation}
Notice that the fixed-point space of $\mathbb{Z}_2 \times \mathbb{Z}_2 $ is
\begin{equation}
\mathcal{U} \bydef \Big\{U \in Y \, : \, U(-t,\Omega) = R_y U(t,\Omega)\text{ and }U(t+\pi,\Omega) = R_z U(t,\Omega)\Big\}.
\end{equation}
This equivariance property implies that we can find solutions by restricting the operator to $\mathcal{U}$,
\[
G : \mathcal{U} \to \mathcal{U}.
\]
\begin{remark}
{\em
Many of the necessary symmetries can be imposed directly in 
coefficient space.  In fact, it can be seen that 
a function $U \in \mathcal{U}$ with Fourier coefficients given by 
\begin{equation}\label{eq:U_sym}
U(t, \Omega) =
%\sum_{k \in even}
\sum_{k \in 2\Z}
\begin{pmatrix}
a_{1,k}(\Omega) \cos kt \\ b_{2,k}(\Omega) \sin kt \\ 0
\end{pmatrix}
+ %\sum_{k \in odd}
\sum_{k \in 2\Z+1}
\begin{pmatrix}
0 \\ 0 \\ a_{3,k}(\Omega) \cos kt
\end{pmatrix},
\end{equation}
has the symmetries of the $P_{12}$ family.  
}
\end{remark}

We will prove that Marchal's $P_{12}$ family begins at the Lagrange triangle with $\Omega = 1$ and ends at the figure eight with $\Omega=0$ by studying periodic solutions $U \in \mathcal{U}$ of \eqref{eq:N_body_delay} for $\Omega \in [0, 1]$.
In particular, we will verify in our proof that we have a figure eight choreography when $\Omega = 0$ as described by the main theorem of \cite{ProofEight}.
The characterization of the eight-shape choreography is reported in the next lemma.

\begin{lemma}\label{lem:eight_shape}
For $\Omega=0$, let $U=(0,U_y,U_z) \in \mathcal{U}$ be a solution of equation \eqref{eq:N_body_delay} with $U_x = 0$.
Then, $q = (q_1, q_2, q_3)$ with $q_j = S^j U$ is a planar choreography solution of the $3$-body problem with the following properties:
\begin{enumerate}
\item[(i)]
The linear momentum of the solution is zero, i.e.
\begin{equation}
\sum_{j=1}^3 q_j (t) = 0.
\end{equation}
\item[(ii)]
The solution in the $yz$-plane translated by $\pi/2$, $\hat{q} (t) \bydef (U_z (t+\pi/2), U_y (t+\pi/2))$, satisfies the symmetries of the figure eight choreography proven in \cite{ProofEight}, i.e.
\begin{align}
\hat{q}(t+\pi) &= (-\hat{q}_{1}(t),\hat{q}_{2}(t)), \\
\hat{q}(-t+\pi) &= (\hat{q}_{1}(t),-\hat{q}_{2}(t)).
\end{align}
\item[(iii)]
The angular momentum is zero, i.e.
\begin{equation}
\sum_{j=1}^3 q_j (t) \times \dot{q}_j (t) = 0.
\end{equation}
\item[(iv)]
If in addition $U \times \dot{U}$ is nowhere vanishing for all $t \in (0, \pi/2)$, then the orbit of $q(t)$ is an eight-shape figure.
\end{enumerate}
\end{lemma}

\begin{proof}
\begin{enumerate}
\item[$(i)$] The conservation of linear momentum implies that
\[
\sum_{j=1}^3 q_j (t) = c t+b.
\]
Since our solution is periodic $c = 0$.
By the symmetries, the $0$-th Fourier modes of $U_y$ and $U_z$, for $U = (U_x, U_y, U_z) \in \mathcal{U}$, are zero.
Also, by assumption $U_x$ is zero, thus the $0$-th Fourier coefficient of $q_j$ is $0$, which implies that $b = 0$.

\item[$(ii)$] For a $yz$-planar solution $U(t) = (0, U_y, U_z) \in \mathcal{U}$ we have that $\hat{q}$ has the symmetries $\hat{q}(t+\pi) = (-\hat{q}_{1}(t),\hat{q}_{2}(t))$ and
\[
\hat{q}(-t+\pi)
= (U_z (-t-\pi/2), U_y (-t-\pi/2))
= (U_z (t+\pi/2), -U_y(t+\pi/2))
= (\hat{q}_1(t), -\hat{q}_2(t)).
\]
These two symmetries agree with the ones found in \cite{ProofEight}.

\item[$(iii)$] By $(ii)$, the symmetries agree with \cite{ProofEight}, see also Remark 3.2 in \cite{MR2012847}.
Using both symmetries one has that $\hat{q}(-t) = -(\hat{q}_1 (t),\hat{q}_2 (t))$, which implies that at time $\pi / 2$ we have the following relations
\begin{align*}
q_1 (\pi/2) &= (0,0,0),\\
q_2 (\pi/2) &= -q_3 (\pi/2) = (0,*,*),\\
\dot{q}_{2}(\pi/2) &= \dot{q}_3 (\pi/2) = (0, *, *),
\end{align*}
where $ (0,*,*)$ denotes a vector whose first component vanishes. 
From these, it is straightforward to verify that the angular momentum is zero at $\pi / 2$.

\item[$(iv)$]It remains to be prove that the shape of the orbit is a figure eight without extra small loops or other unpleasant features.
This fact is an immediate consequence of the corollary following Lemma~7 in \cite{ProofEight}, which requires the assumption in $(iv)$. \qedhere
\end{enumerate}
\end{proof}

\subsection{Desingularization}\label{sec:desingularization}

Marchal's $P_{12}$ family connects to two different families at $\Omega = 1$ and $\Omega = 0$ due to the symmetries of the problem.
At $\Omega=0$ it connects to the family of $x$-translations of the figure eight.
At $\Omega=1$ it connects to the family of homothetic Lagrange triangles.
In this subsection, we augment the system to isolate the $P_{12}$ family from those branches.

\subsubsection{Figure eight}

At $\Omega = 0$, any translation in the $x$-direction of a solution is also solution of \eqref{eq:N_body_delay}.
To isolate the figure eight, we fix the average value of $U_x (t,\Omega)$, with respect to $t$, to be zero along the branch.
To compensate for this restriction, we introduce an unfolding parameter function $\beta=\beta(\Omega)$ in \eqref{eq:N_body_delay}, such that we consider the system
\begin{equation}\label{eq:motion_blowup2}
\beta e_1 - \Omega^2 \bar{I} U - 2 \Omega \bar{J} \dot{U} + \ddot{U} + \sum_{j=1}^2 \frac{U - S^j U}{\| U - S^j U \|^3} = 0,
\end{equation}
with the constraint
\begin{equation}\label{eq:motion_blowup2_condition}
\frac{1}{2\pi} \int_0^{2\pi} U_x (t,\Omega) \, dt = 0,
\qquad \text{for all } \Omega \in [0,1].
\end{equation}

The following lemma shows that periodic solutions of \eqref{eq:motion_blowup2} together with \eqref{eq:motion_blowup2_condition} are equivalent to periodic solutions of \eqref{eq:N_body_delay}.

\begin{lemma}\label{lem:beta}
If $U = (U_x, U_y, U_z) \in \mathcal{U}$ is a periodic solution of \eqref{eq:motion_blowup2} subjected to the condition \eqref{eq:motion_blowup2_condition}, then $\beta(\Omega)=0$ for all $\Omega \in [0, 1]$.
\end{lemma}

\begin{proof}
By the invariance of the $3$-body problem under $x$-translations for $\Omega = 0$, given that $U_j = S^j U$, we have by the conservation of linear momentum in $x$ that
\[
0 = \int_0^{2\pi} \sum_{j=1}^3
\left\langle \ddot{U}_j + \sum_{l \neq j} \frac{U_j - U_l}{\| U_j - U_l \|^3}, e_1 \right\rangle \,dt.
\]
Since this equality still holds for $\Omega \neq 0$, it follows that
\[
0 = \int_0^{2\pi} \sum_{j = 1}^3 S^j
\left\langle \beta e_1 - \Omega^2 \bar{I} U - 2 \Omega \bar{J} \dot{U}, e_1 \right\rangle \,dt
= 6 \pi \beta,
\]
since $\int_0^{2\pi} \dot{U}_{y} \, dt = 0$ by periodicity and $\int_0^{2\pi} U_x \, dt = 0$  by assumption.
This implies $\beta(\Omega) = 0$ for all $\Omega \in [0, 1]$.
\end{proof}

\subsubsection{Lagrange triangle}

At $\Omega=1$, the (planar) homothetic family of the Lagrange equilateral triangle meets the (off-plane) $P_{12}$ family.
The goal of this section is to derive an auxiliary system to \eqref{eq:motion_blowup2} which only retains the $P_{12}$ family.
We use a \emph{blow-up} (as in ``zoom-in'') method (e.g. see \cite{10.1007/s10884-023-10279-x, 10.1137/20M1343464}).
Let $u = (u_1, u_2, u_3)$ be defined by the relation
\begin{equation}
U = L_{\sqrt{a}} u,
\end{equation}
where
\begin{equation}
L_\alpha \bydef
\begin{pmatrix}
1 & 0 & 0\\
0 & 1 & 0\\
0 & 0 & \alpha
\end{pmatrix}.
\end{equation}
Then, the system \eqref{eq:motion_blowup2} becomes
\begin{equation}\label{eq:motion_blowup}
\beta e_1 - \Omega^2 \bar{I} u - 2 \Omega \bar{J} \dot{u} + \ddot{u} + \sum_{j=1}^2 \frac{u-S^j u}{\| L_{\sqrt{a}} u - S^j L_{\sqrt{a}} u \|^3} = 0,
\end{equation}
which we supplement with the conditions
\begin{equation}\label{eq:motion_blowup_condition}
u_3 (0,\Omega) = 1, \qquad
\frac{1}{2\pi} \int_0^{2\pi} u_1 (t,\Omega) \, dt = 0, \qquad
\text{for all } \Omega \in [0, 1].
\end{equation}
For $a \neq 0$, $u = (u_1, u_2, u_3)$ is a periodic solution of \eqref{eq:motion_blowup} if and only if $U = L_{\sqrt{a}} u$ is a periodic solution of \eqref{eq:motion_blowup2}.
The following lemma ensures that the system \eqref{eq:motion_blowup} together with \eqref{eq:motion_blowup_condition} isolates the $P_{12}$ family.

\begin{lemma} \label{lem:triangle}
If $\Omega = 1$ and $a = 0$, in a neighborhood of $u = (u_1, u_2, u_3)$ with
\[
(u_{1},u_{2}) = 3^{-1/6} (\cos2t, -\sin2t),
\]
corresponding to the Lagrange triangle, the equation \eqref{eq:motion_blowup} together with the conditions \eqref{eq:motion_blowup_condition} has an isolated solution $u$.
\end{lemma}

\begin{proof}
Note that the system \eqref{eq:motion_blowup} becomes decoupled in $(u_1, u_2)$ and $u_3$, such that the equation for $(u_1, u_2)$ is solved by the Lagrange triangle, and for $u_3$ we have the linear differential equation
\[
\ddot{u}_3 + C^{-3} \sum_{j=1}^2 u_3 - S^j u_3 = 0,
\]
where $C = \|(u_1, u_2, 0) - S (u_1, u_2, 0)\| = \|(u_1, u_2, 0) - S^2 (u_1, u_2, 0)\| = 3^{1/3}$.
This linear equation admits a family of periodic solutions, parameterized by its amplitude, and determined by the Lagrange triangle $(u_1, u_2)$.
Hence, the linear scaling $u_3 (0, \Omega) = 1$ characterizes the unique periodic solution $u_3 = \cos t$.
\end{proof}

\subsection{Polynomial embedding}

Polynomial nonlinearities benefit inherently from the Banach algebra structure induced by the Fourier-Chebyshev convolutions.
We employ automatic differentiation techniques (e.g. see \cite{Bucker2006,MR2531684,MR3973675,Henot2021,Jorba2005,Knuth1981,Lessard2016}) to construct an auxiliary system to \eqref{eq:motion_blowup} only comprised of polynomial nonlinearities.

We introduce
\[
v = \partial_t u, \qquad w_j = \| L_{\sqrt{a}} u - S^j L_{\sqrt{a}} u \|^{-1}, \quad \text{for } j = 1, 2.
\]
Notice that the symmetries of $u \in \mathcal{U}$ extend naturally to $v$ and $w$
\begin{align*}
v(t+\pi, \Omega) &= R_z v(t, \Omega), \qquad v(-t, \Omega)=-R_{y}v(t, \Omega), \\
w_j (t+\pi, \Omega) &= w_j (t, \Omega),\qquad w_j (-t, \Omega)=w_{3-j} (t, \Omega).
\end{align*}
The only nontrivial identity is the last one, which follows from
\begin{align*}
w_2 (-t, \Omega)
&= \frac{1}{\left\| L_{\sqrt{a}} u (-t, \Omega) - L_{\sqrt{a}} u(-(t-8\pi/3), \Omega) \right\|} \\
&= \frac{1}{\left\| R_y L_{\sqrt{a}} u(t, \Omega) - R_y L_{\sqrt{a}} u(t-4\pi/3, \Omega) \right\|}
= w_1 (t, \Omega).
\end{align*}
Given a function $h \in C^\infty(\mathbb{R}/2\pi\mathbb{Z} \times [0, 1], \mathbb{R})$, we define
\begin{equation}
[Rh](t, \Omega) \bydef h(-t, \Omega).
\end{equation}
With this notation, the above identity reads $w_2 = Rw_1$.
To avoid carrying out an additional subscript, we let $w = w_1$, so that $w_2 = Rw$.

Introducing an unfolding parameter $\alpha=\alpha(\Omega)$, the system \eqref{eq:motion_blowup} can be re-written as the first-order polynomial system
\begin{equation}\label{eq:motion_blowup_DDE}
\begin{aligned}
\partial_t v &= - \beta e_1 + \Omega^2 \bar{I} u + 2 \Omega \bar{J} v - w^3 (u - S  u) - R w^3 (u - S^2 u), \\
\partial_t u &= v, \\
\partial_t w &= - \alpha - w^3 \big\langle u - S  u, L_a (v - S  v) \big\rangle,
\end{aligned}
\end{equation}
which we supplement with the conditions
\begin{equation}\label{eq:motion_blowup_DDE_conditions}
u_3 (0,\Omega) = 1, \quad
\frac{1}{2\pi} \int_0^{2\pi} u_1 (t,\Omega) \, dt = 0, \quad
\Big[w^2 \big\langle u - Su, L_a (u - Su) \big\rangle \Big](0,\Omega) = 1, \quad
\text{for all } \Omega \in [0, 1],
\end{equation}
and where
\[
w^3 (u - Su) =
\begin{pmatrix}
w^3 (u_1 - Su_1)\\
w^3 (u_2 - Su_2)\\
w^3 (u_3 - Su_3)
\end{pmatrix}.
\]

The following lemma shows that periodic solutions of \eqref{eq:motion_blowup_DDE} together with \eqref{eq:motion_blowup_DDE_conditions} are equivalent to periodic solutions of \eqref{eq:motion_blowup}.

\begin{lemma}\label{lem:alpha}
If $(u, v, w)$ is a periodic solution of \eqref{eq:motion_blowup_DDE} subjected to the conditions \eqref{eq:motion_blowup_DDE_conditions}, then $\alpha(\Omega) = 0$ for all $\Omega \in [0,1]$.
\end{lemma}

\begin{proof}
Let us first prove that $w$ can never vanish.
Indeed, suppose that $w(t_0, \Omega) = 0$ for some $t_0$.
If $\alpha = 0$, then $w = 0$ by uniqueness of the initial value problem, which violates $\Big[w^2 \big\langle u - Su, L_a (u-Su) \big\rangle\Big](0,\Omega) = 1$.
On the other hand, if $\alpha \ne 0$, then $w$ changes sign at $t_0$.
Since $\alpha$ is a strictly positive or negative constant, $w$ can never change sign again which contradicts the fact that $w$ is periodic.

Therefore $w \neq 0$, from which we deduce that
\[
w = \left( \big\langle u - Su, L_a (u-Su) \big\rangle + C + 2 \alpha \int_0^t w(t')^{-3}\,dt' \right)^{-1/2},
\]
where $C$ is a constant of integration.
Since $w$ is either strictly positive or negative, the periodicity of $w$ requires $\alpha = 0$.
\end{proof}

\subsection{Zero-finding problem}

We now write \eqref{eq:motion_blowup_DDE} together with \eqref{eq:motion_blowup_DDE_conditions} as the zero of a mapping $F$.
Specifically, we consider the mapping given by
\begin{equation}\label{eq:zero-finding-problem}
F(a,\beta,\alpha,u,v,w) \bydef
\begin{pmatrix}
\eta(u) \\ \gamma(a,u,w) \\ g(\beta,u,v,w) \\ f(u,v) \\ h(a,\alpha,u,v,w)
\end{pmatrix},
\end{equation}
where
\begin{align}
\eta(u) &\bydef \Big(u_3 (0, \Omega) - 1, \frac{1}{2\pi} \int_0^{2\pi} u_1(t, \Omega) \, dt \Big),
\\
\gamma(a, u, w) &\bydef \Big[w^2 \big\langle u - S u, L_a (u - S u) \big\rangle\Big](0, \Omega) - 1,
\\
g(\beta, u, v, w) &\bydef \partial_t v  + \beta e_1 - \Omega^2 \bar{I} u - 2 \Omega \bar{J} v + w^3 (u - S  u) + R w^3 (u - S^2  u),
\\
f(u, v) &\bydef \partial_t u - v,
\\
h(a, \alpha, u, v, w) &\bydef \partial_t w + \alpha + w^3 \big\langle u - S  u, L_a (v - S  v) \big\rangle.
\end{align}

The exact domain and image of $F$ is detailed in Section \ref{sec:proof}.
For the time being, we define the spaces
\begin{equation}\label{eq:V_space}
\mathcal{V} \bydef \Big\{v \in Y \, : \, v(-t,\Omega) = -R_y v(t,\Omega) \text{ and } v(t+\pi, \Omega) = R_z v(t,\Omega)\Big\},
\end{equation}
and
\begin{equation}\label{eq:W_space}
\mathcal{W} \bydef \Big\{w  \in C^\infty(\mathbb{R}/2\pi\mathbb{Z} \times [0,1], \mathbb{R}) \, : \, w(t+\pi, \Omega) = w(t,\Omega)\Big\}.
\end{equation}

\begin{lemma}\label{lem:well_def}
The map $F:C^\infty ([0, 1], \R)^3\times\mathcal{U}\times\mathcal{V}\times\mathcal{W} \to C^\infty ([0, 1], \R)^3\times\mathcal{U}\times\mathcal{V}\times\mathcal{W}$ is well-defined.
\end{lemma}

\begin{proof}
We need to verify that $F(a,\beta,\alpha,u,v,w) \in C^\infty ([0, 1], \R)^3 \times \mathcal{U} \times \mathcal{V} \times \mathcal{W}$ if $(a,\beta,\alpha,u,v,w) \in C^\infty ([0, 1], \R)^3 \times \mathcal{U} \times \mathcal{V} \times \mathcal{W}$.
The only term that requires some computation is the nonlinear term
\[
P(u,w) \bydef w^3 ( u - S u ) + R w^3 ( u - S^2 u ).
\]
The fact that $P(u,w) \in \mathcal{U}$ if $(u, w) \in \mathcal{U} \times \mathcal{W}$ follows from
\begin{align*}
[P(u,w)](-t, \Omega)
&= w^3 (-t, \Omega) \big( u(-t, \Omega) - u(-t+4\pi/3, \Omega) \big) + w^3(t, \Omega) \big( u(-t, \Omega) - u(-t-4\pi/3, \Omega) \big) \\
&= w^3 (-t, \Omega) \big( R_y u(t, \Omega) - R_y u(t-4\pi/3, \Omega) \big) + w^3(t, \Omega) \big( R_y u(t, \Omega) - R_y u(t+4\pi/3, \Omega) \big) \\
&= R_y [P(u,w)](t, \Omega). \qedhere
\end{align*}
\end{proof}

%!TEX root = marchal.tex

\section{Existence of a branch of solutions of the functional equation}
\label{sec:proof}

In this section, we prove the existence of a one-parameter family of zeros of $F$, defined formally in \eqref{eq:zero-finding-problem}, parameterized by $\Omega \in [0, 1]$.
As a matter of fact, we consider the complexification of the spaces $\mathcal{U}$, $\mathcal{V}$ and $\mathcal{W}$, respectively denoted by $\mathcal{U}^\mathbb{C}$, $\mathcal{V}^\mathbb{C}$ and $\mathcal{W}^\mathbb{C}$.
The mapping $F : C^\infty ([0, 1], \C)^3 \times \mathcal{U}^\mathbb{C} \times \mathcal{V}^\mathbb{C} \times \mathcal{W}^\mathbb{C} \to C^\infty ([0, 1], \C)^3 \times \mathcal{U}^\mathbb{C} \times \mathcal{V}^\mathbb{C} \times \mathcal{W}^\mathbb{C}$ is still well-defined, as in Lemma \ref{lem:well_def}.
In Section~\ref{sec:aposteriori}, we conclude the proof of Marchal's conjecture by verifying that this zero is indeed in $C^\infty ([0, 1], \R)^3 \times \mathcal{U} \times \mathcal{V} \times \mathcal{W}$ and that it represents the branch of the $P_{12}$ family, joining the Lagrange triangle to the figure eight.

\subsection{Choice of norms}

Let $\nu \ge 1$.
Consider the Banach space
\begin{equation}
\mathcal{U}_\nu \bydef
\Big\{
u \in \mathcal{U}^\mathbb{C} \, : \, \| u \|_{\mathcal{U}_\nu } \bydef  \sum_{j=1}^3 \| u_j \|_\nu < \infty
\Big\}
\end{equation}
where, given $\phi \in C^\infty (\R/2\pi\Z \times [0, 1], \C)$,
\begin{equation} \label{eq:nu_mu_norm}
\| \phi \|_\nu \bydef \frac{1}{\pi^2}  \sum_{k,n \in \Z} 
 \left| \int_0^{2\pi} \hspace{-.2cm} \int_0^1 \hspace{-.1cm} \frac{\phi(t, \Omega) e^{ikt} \mathcal{T}_{|n|} (2\Omega - 1)}{\sqrt{1 - (2\Omega - 1)^2}} \, d\Omega dt \right| \nu^{|k|}.
\end{equation}
Here, $\mathcal{T}_n : [-1, 1] \to [-1, 1]$ represents the $n$-th Chebyshev polynomial of the first kind, given by the recurrence relation
\begin{equation}
\mathcal{T}_0(s) = 1, \qquad \mathcal{T}_1(s) = s, \qquad \mathcal{T}_n(s) = 2 s \mathcal{T}_{n-1} (s) - \mathcal{T}_{n-2}(s), \quad n \ge 2.
\end{equation}
Importantly, these polynomials satisfy the identity $\mathcal{T}_n (\cos \theta) = \cos (n \theta)$, meaning that infinite series of Chebyshev polynomials amounts to cosine series.
This highlights the significance of employing them as a basis for series expansions with respect to the frequency $\Omega$, since the convergence property of analytic Fourier series holds true for functions defined on the entire range of parameter values $[0, 1]$.
This fact contrasts with Taylor expansions, where despite the function's analyticity, the presence of poles frequently necessitates partitioning the function's domain into numerous subintervals.

Consider also
\begin{align}
\mathcal{V}_\nu &\bydef
\Big\{
v \in \mathcal{V}^\mathbb{C} \, : \, \| v \|_{\mathcal{V}_\nu} \bydef \sum_{j=1}^3 \| v_j \|_\nu < \infty
\Big\}, \\
\mathcal{W}_\nu &\bydef
\Big\{
w \in \mathcal{W}^\mathbb{C} \, : \, \| w\|_\nu < \infty
\Big\}, \\
X &\bydef
\Big\{
\psi : C^\infty ([0, 1], \C) \, : \, \| \psi \|_X < \infty
\Big\},
\end{align}
where, given $\psi \in C^\infty ([0, 1], \C)$,
\begin{equation}
\| \psi \|_X  \bydef \frac{2}{\pi} \sum_{n \in \Z} \left|\int_0^1 \frac{\psi(\Omega) \mathcal{T}_{|n|} (2\Omega - 1)}{\sqrt{1 - (2\Omega - 1)^2}} \, d\Omega \right|.
\end{equation}
It will be convenient for the forthcoming norm estimates to observe that the spaces $\mathcal{U}_\nu, \mathcal{V}_\nu, \mathcal{W}_\nu, X$ form unital Banach algebras with respect to the product of functions.
Let us detail the demonstration for $X$, namely let us show, for any $\psi_1, \psi_2 \in X$, that $\| \psi_1 \psi_2 \|_X \le \| \psi_1 \|_X \| \psi_2 \|_X$:
\begin{align*}
\| \psi_1 \psi_2 \|_X
&= \frac{2}{\pi} \sum_{n \in \Z} \left|\int_0^1 \frac{\psi_1(\Omega)\psi_2(\Omega) \mathcal{T}_{|n|} (2\Omega - 1)}{\sqrt{1 - (2\Omega - 1)^2}} \, d\Omega \right| \\
&= \frac{1}{\pi} \sum_{n \in \Z} \left|\int_0^{\pi} \psi_1\Big(\frac{1+\cos \theta}{2}\Big) \psi_2\Big(\frac{1+\cos \theta}{2}\Big) \cos (n \theta) \, d\theta \right| \\
&= \frac{1}{\pi^2} \sum_{n \in \Z} \left| \sum_{n' \in \Z} \left(\int_0^{\pi} \psi_1\Big(\frac{1+\cos \theta}{2}\Big) \cos ((n-n') \theta) \, d\theta\right) \left(\int_0^{\pi} \psi_2\Big(\frac{1+\cos \theta}{2}\Big) \cos (n' \theta) \, d\theta\right) \right| \\
&\le \left(\frac{1}{\pi} \sum_{n \in \Z} \left|\int_0^{\pi} \psi_1(\cos \theta) \cos (n \theta) \, d\theta \right| \right) \left( \frac{1}{\pi} \sum_{n \in \Z} \left|\int_0^{\pi} \psi_2(\cos \theta) \cos (n \theta) \, d\theta \right| \right)\\
&= \| \psi_1 \|_X \| \psi_2 \|_X.
\end{align*}

Finally, define
\begin{equation}\label{eq:banach_space}
\cX_\nu \bydef X \times X \times X \times \mathcal{U}_\nu \times \mathcal{V}_\nu \times \mathcal{W}_\nu,
\end{equation}
endowed with the norm
\begin{equation}
\| x \|_{\cX_\nu} \bydef \| a \|_X + \| \beta \|_X + \| \alpha \|_X + \| u \|_{\mathcal{U}_\nu} + \| v \|_{\mathcal{V}_\nu} + \| w \|_\nu,
\qquad \text{for all } x = (a, \beta, \alpha, u, v, w) \in \cX_\nu.
\end{equation}

Denoting by $\mathcal{D}(F)$ the domain of $F$, we obtain the following.

\begin{proposition}
\[
F : \mathcal{D}(F) \subset \cX_\nu \to \cX_\nu.
\]
\end{proposition}

\begin{proof}
This follows from Lemma \ref{lem:well_def} (adapted to the complexified space) and
\[
\mathcal{D}(F) =
\Big\{
x = (a, \beta, \alpha, u, v, w) \in \cX_\nu \, : \,
\|\partial_t u\|_{\mathcal{U}_\nu} < \infty, \|\partial_t v\|_{\mathcal{V}_\nu} < \infty, \|\partial_t w\|_\nu < \infty
\Big\}. \qedhere
\]
\end{proof}

\subsection{Verification of the contraction hypotheses}

We now establish the contraction of a Newton-type operator around a numerical approximation of the branch of solutions.
Without deviating from the focus of this manuscript, it is important to highlight the efficacy of our strategy.
Most existing rigorous continuation techniques employ piecewise-linear -- or even piecewise-constant -- approximations of the branch, significantly limiting the step-size between parameter values and incurring substantial computational cost.
In contrast, our technique, capable of obtaining the entire branch of solutions as a single fixed-point, is based on the recent works \cite{Breden,OlivierBreden}.
In essence, the argument leverages the algebraic structure of $F$, with respect to the parameter $\Omega$, to verify the contraction mapping theorem around a high-order approximation of the branch.

The $k$-th Fourier coefficient of a function $\phi \in C^\infty (\R/2\pi\Z \times [0, 1], \mathbb{C})$ is denoted by a subscript $k$ as follows
\[
\phi_k (\Omega) = \frac{1}{2\pi} \int_0^{2\pi} \phi(t, \Omega) e^{ikt} \, dt,
\qquad \Omega \in [0, 1], \quad k \in \Z,
\]
so that $\phi_k \in X$ for all $k \in \Z$.
The Fourier-Chebyshev coefficients are denoted by the double subscript $n,k$ as follows
\[
\phi_{n,k} = \frac{1}{\pi^2} \int_0^{2\pi} \int_0^1 \frac{\phi(t, \Omega) e^{ikt} \mathcal{T}_n (2\Omega - 1)}{\sqrt{1 - (2\Omega - 1)^2}} \, d\Omega dt,
\qquad n \in \N, \quad k \in \Z.
\]
Lastly, using a single subscript $n$, we denote the $n$-th Chebyshev coefficient of a function $\psi \in C^\infty ([0, 1], \mathbb{C})$, namely
\[
\psi_n = \frac{2}{\pi} \int_0^1 \frac{\psi(\Omega) \mathcal{T}_n (2\Omega - 1)}{\sqrt{1 - (2\Omega - 1)^2}} \, d\Omega,
\qquad n \in \N.
\]

Our goal is to prove that a fixed-point operator $T$ (yet to be constructed) is contracting around a finite Fourier-Chebyshev series approximating the $P_{12}$ family.
This approximation lives in a finite dimensional subspace of $\cX_\nu$, which we now detail.
Define the \emph{truncation operators} $\Pi_K, \Pi_{N, K} : C^\infty (\R/2\pi\Z \times [0, 1], \mathbb{C}) \to C^\infty (\R/2\pi\Z \times [0, 1], \mathbb{C})$ by
\begin{equation}
(\Pi_K u)_k \bydef
\begin{cases}
u_k, & |k| \le K, \\
0, & |k| > K,
\end{cases} \qquad
(\Pi_{N,K} u)_{n,k} \bydef
\begin{cases}
(\hat \Pi_N u_k)_n, & |k| \le K, \\
0, & |k| > K,
\end{cases}
\end{equation}
with $\hat \Pi_N : X \to X$ given by
\begin{equation}
(\hat \Pi_N u)_n \bydef
\begin{cases}
u_n, & n \le N, \\
0, & n > N.
\end{cases}
\end{equation}
Consider also the \emph{tail operator} $\Pi_{\infty(K)} \bydef I - \Pi_K$.
Both truncation operators $\Pi_K$ and $\Pi_{N, K}$ naturally extend to all $x = (a, \beta, \alpha, u, v, w) \in \cX_\nu$ by acting component-wise
\[
\Pi_K x = (a, \beta, \alpha, \Pi_K u, \Pi_K v, \Pi_K w), \qquad \Pi_{N, K} x = (\hat \Pi_N a, \hat \Pi_N \beta, \hat \Pi_N \alpha, \Pi_{N, K} u, \Pi_{N, K} v, \Pi_{N, K} w).
\]
Fixing $N, K > 0$, one may construct an interpolation of the zeros of $F$ in $\Pi_{N,K} \cX_\nu$, denoted by $\bar{x}$.
For the sake of completeness, we detail the procedure we followed for this proof.
For each $j = 0, \dots, N$, we find approximate unfolding parameters $a, \beta, \alpha \in \mathbb{C}$ and $2\pi$-periodic solution $u, v \in C^\infty (\R/2\pi\Z, \mathbb{C}^3), w \in C^\infty (\R/2\pi\Z, \mathbb{C})$ to \eqref{eq:motion_blowup_DDE} satisfying \eqref{eq:motion_blowup_DDE_conditions} at the Chebyshev nodes
\[
\Omega_j \bydef \frac{1}{2}(1 + \cos(j\pi/N)).
\]
To be concrete, the periodic solution $u(t, \Omega_j), v(t, \Omega_j), w(t, \Omega_j)$ are represented by finite Fourier series, with the appropriate symmetries of the $P_{12}$ family as described in \eqref{eq:U_sym}:
\begin{alignat*}{5}
&u_1(t, \Omega_j) = \sum_{k \in \Z} (u_1)_k (\Omega_j) e^{i k t}, \qquad &&(u_1)_k = &&(u_1)_{-k}, \quad &&(u_1)_{2k+1} &&= 0,
\\
&u_2(t, \Omega_j) = \sum_{k \in \Z} (u_2)_k (\Omega_j) e^{i k t}, \qquad &&(u_2)_k = -&&(u_2)_{-k}, \quad &&(u_2)_{2k} &&= 0,
\\
&u_3(t, \Omega_j) = \sum_{k \in \Z} (u_3)_k (\Omega_j) e^{i k t}, \qquad &&(u_3)_k = &&(u_3)_{-k}, \quad &&(u_3)_{2k+1} &&= 0,
\end{alignat*}
and similarly for $v$ and $w$.
Hence, we solve numerically $N+1$ finite dimensional problems by using Newton's method on $\Pi_{0,K} F|_{\Omega_j} \circ \Pi_{0,K}$, and we denote by $\bar{x}_j \in \Pi_{0,K} \cX_\nu$ such an approximate zero, i.e. $\Pi_{0,K} F|_{\Omega_j} (\bar{x}_j) \approx 0$.
Then, we retrieve numerically the Chebyshev polynomial associated to each component of $\bar{x}_j$ via the \emph{inverse discrete Fourier transform}, thereby yielding $\bar{x} \in \Pi_{N, K} \cX_\nu$.
It is important to note that for Theorem~ \ref{thm:main} to apply, we ensure that the numerical branch of choreographies precisely passes through the Chebyshev nodes, with particular emphasis on the nodes $\Omega_j \in \{0,1\}$. The resulting Fourier-Chebyshev series coefficients are found in \cite{Henot2024}. Incidentally, it is practical to opt for an appropriate number of Chebyshev nodes to take full advantage of the \emph{fast Fourier transform} algorithm.

Given two Banach spaces $\mathfrak{X}, \mathfrak{Y}$, we denote by $\mathscr{B}(\mathfrak{X}, \mathfrak{Y})$ the set of bounded linear operators from $\mathfrak{X}$ to $\mathfrak{Y}$.
To prove the existence of a zero of $F$, let us construct $A \in \mathscr{B}(\cX_\nu, \cX_\nu)$ so that
\begin{equation}\label{eq:fixed-point-operator}
T(x) \bydef x - A F(x)
\end{equation}
satisfies the Banach Fixed-Point Theorem.
The injectivity of $A$, required to obtain a zero of $F$, will be a by-product of the contraction.

We define the bounded linear operator $A : \cX_\nu \to \cX_\nu$ given by
\[
A = A_\textnormal{finite} \Pi_K + A_\textnormal{tail} \Pi_{\infty(K)},
\]
where $A_\textnormal{finite} = A_\textnormal{finite} (\bar{x}) \in \mathscr{B}(\Pi_K X, \Pi_K X)$ is an approximation of $(\Pi_{K} DF(\bar{x}) \Pi_{K})^{-1}$ and $A_\textnormal{tail} : \Pi_{\infty(K)} \cX_\nu \to \Pi_{\infty(K)} \cX_\nu$ is defined, for all $x = (a, \beta, \alpha, u, v, w) \in \cX_\nu$, by
\begin{align*}
A_\textnormal{tail} x \bydef (0, 0, 0, u', v', w'), \qquad
(u_j')_k &\bydef
\begin{cases}
0, & |k| \le K, \\
(ik)^{-1} (u_j)_k, & |k| > K,
\end{cases} \quad j = 1, 2, 3,\\
(v_j')_k &\bydef
\begin{cases}
0, & |k| \le K, \\
(ik)^{-1} (v_j)_k, & |k| > K,
\end{cases} \quad j = 1, 2, 3,\\
(w')_k &\bydef
\begin{cases}
0, & |k| \le K, \\
(ik)^{-1} w_k, & |k| > K.
\end{cases}
\end{align*}
We also emphasize that computing $A_\textnormal{finite}$ is done without numerically inverting the entire square matrix representing $\Pi_{N,K} DF(\bar{x}) \Pi_{N,K}$.
Since $DF(\bar{x})$ amounts to a multiplication operator with respect to $\Omega$, we compute numerically $A_j \in \mathscr{B}(\Pi_{0, K} \cX_\nu, \Pi_{0, K} \cX_\nu)$ such that $A_j ( \Pi_{0, K} DF|_{\Omega_j}(\bar{x}_j) \Pi_{0, K} ) \approx \Pi_{0, K}$ and, as done for $\bar{x}$, we use the inverse discrete Fourier transform to obtain the Chebyshev polynomial associated to each component of $A_j$.

By construction of $A$, we have that $T \in C^2 (\cX_\nu)$.
For the remainder of this paper, we fix
\begin{equation}
K = 70, \qquad N = 20, \qquad \nu = \frac{11}{10}, \qquad r = 10^{-6}.
\end{equation}
We now give three lemmata leading to the proof that $T$ is a contraction in $\textnormal{cl}(B_r (\bar{x}))$.

\begin{lemma}\label{lem:Y}
\[
\| A F (\bar{x}) \|_{\cX_\nu} \le 10^{-8}.
\]
\end{lemma}

\begin{proof}
We have that
\begin{align*}
\| A F(\bar{x}) \|_{\cX_\nu}
&= \| A_\textnormal{finite} \Pi_K F(\bar{x}) + A_\textnormal{tail} \Pi_{\infty(K)} F(\bar{x}) \|_{\cX_\nu} \\
&\le \|A_\textnormal{finite} \Pi_K F(\bar{x})\|_{\cX_\nu} + \frac{1}{K+1} \left( \| \Pi_{\infty(K)} g(\bar{\beta}, \bar{u}, \bar{v}, \bar{w}) \|_{\mathcal{V}_\nu} + \| \Pi_{\infty(K)} h(\bar{a}, \bar{\alpha}, \bar{u}, \bar{v}, \bar{w}) \|_\nu\right).
\end{align*}
Since $g$ and $h$ are polynomials, we can find an upper bound for $\| A F(\bar{x}) \|_{\cX_\nu}$ by a finite number of computations, which are carried out by a computer with interval arithmetic.
\end{proof}

\begin{lemma}\label{lem:Z1}
\[
\| I - A D F (\bar{x}) \|_{\mathscr{B}(\cX_\nu, \cX_\nu)} \le \frac{9}{10}.
\]
\end{lemma}

\begin{proof}
We have that
\[
\| I - A DF(\bar{x}) \|_{\mathscr{B}(\cX_\nu, \cX_\nu)}
= \max \Big(\| \Pi_{2K} - A DF(\bar{x})\Pi_{2K} \|_{\mathscr{B}(\cX_\nu, \cX_\nu)}, \| \Pi_{\infty(2K)} - A DF(\bar{x})\Pi_{\infty(2K)} \|_{\mathscr{B}(\cX_\nu, \cX_\nu)} \Big).
\]
On the one hand,
{\small
\begin{align*}
&\| \Pi_{2K} - A DF(\bar{x})\Pi_{2K} \|_{\mathscr{B}(\cX_\nu, \cX_\nu)} \\
&\le \| \Pi_K - A_\textnormal{finite} \Pi_K DF(\bar{x}) \Pi_{2K} \|_{\mathscr{B}(\cX_\nu, \cX_\nu)} \\
&\hspace{0.5cm} + \frac{1}{K+1}
\max \Big(
\|\Pi_{\infty(K)} \partial_a h(\bar{a}, \bar{\alpha}, \bar{u}, \bar{v}, \bar{w})\|_{\mathscr{B}(X, \mathcal{C}_\nu)}, \\
&\hspace{1.5cm} \max_{i=1,2,3} \Big(\sum_{j=1}^3 \| \partial_{u_i} g_j(\bar{\beta}, \bar{u}, \bar{v}, \bar{w}) \|_{\mathscr{B}(\mathcal{C}_\nu, \mathcal{C}_\nu)} + \| \partial_{u_i} h(\bar{a}, \bar{\alpha}, \bar{u}, \bar{v}, \bar{w}) \|_{\mathscr{B}(\mathcal{C}_\nu, \mathcal{C}_\nu)}\Big), \\
&\hspace{1.5cm} \max_{i=1,2,3} \Big(\sum_{j=1}^3 \| \partial_{v_i} g_j(\bar{\beta}, \bar{u}, \bar{v}, \bar{w}) - \delta_{ji} \partial_t \|_{\mathscr{B}(\mathcal{C}_\nu, \mathcal{C}_\nu)} + \| \partial_{v_i} f_j(\bar{u}, \bar{v}) \|_{\mathscr{B}(\mathcal{C}_\nu, \mathcal{C}_\nu)} + \| \partial_{v_i} h(\bar{a}, \bar{\alpha}, \bar{u}, \bar{v}, \bar{w}) \|_{\mathscr{B}(\mathcal{C}_\nu, \mathcal{C}_\nu)}\Big), \\
&\hspace{2.65cm} \sum_{j=1}^3 \| \partial_w g_j(\bar{\beta}, \bar{u}, \bar{v}, \bar{w}) \|_{\mathscr{B}(\mathcal{C}_\nu, \mathcal{C}_\nu)} + \| \partial_w h(\bar{a}, \bar{\alpha}, \bar{u}, \bar{v}, \bar{w}) - \partial_t \|_{\mathscr{B}(\mathcal{C}_\nu, \mathcal{C}_\nu)}
\Big) \\
&\le \| \Pi_K - A_\textnormal{finite} \Pi_K DF(\bar{x}) \Pi_{2K} \|_{\mathscr{B}(\cX_\nu, \cX_\nu)} \\
&\hspace{0.5cm} + \frac{1}{K+1}
\max \Big( 
\|\Pi_{\infty(K)} \bar{w}^3 \big\langle \bar{u}_3 - S \bar{u}_3, \bar{v}_3 - S \bar{v}_3 \big\rangle\|_\nu, \\
&\hspace{3cm} \max_{i=1,2,3} \Big(\|\Omega^2\|_X (\delta_{1i} + \delta_{2i}) + \sqrt{3}\big(2\|\bar{w}^3\|_\nu + \|\bar{w}^3 (L_{\bar{a}} (\bar{v} - S \bar{v}))_i\|_\nu\big)\Big), \\
&\hspace{3cm} 1 + \max_{i=1,2,3} \Big(\|2\Omega\|_X (\delta_{1i} + \delta_{2i}) + \sqrt{3} \|\bar{w}^3 (L_{\bar{a}} (\bar{u} - S \bar{u}))_i\|_\nu\Big), \\
&\hspace{3cm} \|3\bar{w}^2 \big\langle \bar{u} - S \bar{u}, L_{\bar{a}} (\bar{v} - S \bar{v}) \big\rangle\|_\nu + \sum_{i=1}^3 \|3\bar{w}^2 (\bar{u}_i - S \bar{u}_i)\|_\nu + \|3\bar{w}^2 (\bar{u}_i - S^2 \bar{u}_i)\|_\nu
\Big),
\end{align*}
}
where
\begin{equation}
\mathcal{C}_\nu \bydef
\Big\{\phi \in C^\infty(\R/2\pi\Z \times [0, 1], \C) \, : \, \| \phi \|_\nu < \infty\Big\}.
\end{equation}
We also slightly abuse notation and write $\| \psi(\Omega) \|_X$ to denote $\|\psi\|_X$ for $\psi \in X$, e.g. $\|\Omega^2\|_X = \| \psi\|_X$ with $\psi(\Omega) = \Omega^2$.
Moreover, the symbol $\delta_{ji}$ denotes the Kronecker delta, and we use the fact that $\|R\|_{\mathscr{B}(\mathcal{C}_\nu, \mathcal{C}_\nu)} = 1$, $\|I - S\|_{\mathscr{B}(\mathcal{C}_\nu, \mathcal{C}_\nu)} = \|I - S^2\|_{\mathscr{B}(\mathcal{C}_\nu, \mathcal{C}_\nu)} = \sqrt{3}$.

On the other hand, the multiplication operator associated with $\phi \in C^\infty (\R/2\pi\Z \times [0, 1], \mathbb{C})$, and given by $\mathcal{M}_\phi \psi  \bydef \phi \psi$ for all $\psi \in C^\infty (\R/2\pi\Z \times [0, 1], \mathbb{C})$, satisfies
\[
\big(\Pi_K \mathcal{M}_\phi \Pi_{\infty(2K)}\big)_{k,k'} = \big(\Pi_{\infty(K)} \phi\big)_{k' - k}, \qquad k,k' \in \Z.
\]
Hence,
{\footnotesize
\begin{align*}
&\| \Pi_{\infty(2K)} - A DF(\bar{x})\Pi_{\infty(2K)} \|_{\mathscr{B}(\cX_\nu, \cX_\nu)} \\
&\le \|A_\textnormal{finite}\|_{\mathscr{B}(\cX_\nu, \cX_\nu)} \\
&\hspace{0cm} \times \max\Big(
\max_{i=1,2,3} \Big( \|\partial_{u_i} \eta_1(\bar{u}) \Pi_{\infty(2K)}\|_{\mathscr{B}(\mathcal{C}_\nu, X)} + \|\partial_{u_i} \gamma(\bar{a}, \bar{u}, \bar{w}) \Pi_{\infty(2K)}\|_{\mathscr{B}(\mathcal{C}_\nu, X)} + \sum_{j=1}^3 \|\Pi_{\infty(K)}\partial_{u_i} g_j(\bar{\beta}, \bar{u}, \bar{v}, \bar{w})\|_{\mathscr{B}(\mathcal{C}_\nu, \mathcal{C}_\nu)} \\
&\hspace{10.33cm} + \|\Pi_{\infty(K)}\partial_{u_i} h(\bar{a}, \bar{\alpha}, \bar{u}, \bar{v}, \bar{w})\|_{\mathscr{B}(\mathcal{C}_\nu, \mathcal{C}_\nu)} \Big), \\
&\hspace{1.15cm} \max_{i=1,2,3} \|\Pi_{\infty(K)}\partial_{v_i} h(\bar{a}, \bar{\alpha}, \bar{u}, \bar{v}, \bar{w})\|_{\mathscr{B}(\mathcal{C}_\nu, \mathcal{C}_\nu)}, \\
&\hspace{1.33cm} \|\partial_w \gamma(\bar{a}, \bar{u}, \bar{w}) \Pi_{\infty32K)}\|_{\mathscr{B}(\mathcal{C}_\nu, X)} + \sum_{j=1}^3 \|\Pi_{\infty(K)}\partial_w g_j(\bar{\beta}, \bar{u}, \bar{v}, \bar{w})\|_{\mathscr{B}(\mathcal{C}_\nu, \mathcal{C}_\nu)} + \|\Pi_{\infty(K)}\partial_w h(\bar{a}, \bar{\alpha}, \bar{u}, \bar{v}, \bar{w})\|_{\mathscr{B}(\mathcal{C}_\nu, \mathcal{C}_\nu)}
\Big) \\
&\hspace{0cm} + \frac{1}{K+1} \max \Big( 
\max_{i=1,2,3} \Big(\sum_{j=1}^3 \| \partial_{u_i} g_j(\bar{\beta}, \bar{u}, \bar{v}, \bar{w}) \|_{\mathscr{B}(\mathcal{C}_\nu, \mathcal{C}_\nu)} + \| \partial_{u_i} h(\bar{a}, \bar{\alpha}, \bar{u}, \bar{v}, \bar{w}) \|_{\mathscr{B}(\mathcal{C}_\nu, \mathcal{C}_\nu)} \Big), \\
&\hspace{2.07cm} \max_{i=1,2,3} \Big(\sum_{j=1}^3 \| \partial_{v_i} g_j (\bar{\beta}, \bar{u}, \bar{v}, \bar{w}) - \delta_{ji} \partial_t \|_{\mathscr{B}(\mathcal{C}_\nu, \mathcal{C}_\nu)} + \| \partial_{v_i} f_j (\bar{u}, \bar{v}) \|_{\mathscr{B}(\mathcal{C}_\nu, \mathcal{C}_\nu)} + \| \partial_{v_i} h(\bar{a}, \bar{\alpha}, \bar{u}, \bar{v}, \bar{w}) \|_{\mathscr{B}(\mathcal{C}_\nu, \mathcal{C}_\nu)} \Big), \\
&\hspace{3.2cm} \sum_{j=1}^3 \| \partial_w g_j(\bar{\beta}, \bar{u}, \bar{v}, \bar{w}) \|_{\mathscr{B}(\mathcal{C}_\nu, \mathcal{C}_\nu)} + \| \partial_w h(\bar{a}, \bar{\alpha}, \bar{u}, \bar{v}, \bar{w}) - \partial_t \|_{\mathscr{B}(\mathcal{C}_\nu, \mathcal{C}_\nu)}
\Big) \\
& \le \|A_\textnormal{finite}\|_{\mathscr{B}(\cX_\nu, \cX_\nu)} \\
&\hspace{0.5cm} \times \max\Big( 
\sqrt{3} \max_{i=1,2,3} \Big( \big(\delta_{3i} + \|[2 \bar{w}^2 ( L_{\bar{a}} (\bar{u} - S \bar{u}) )_i ](0, \Omega)\|_X\big)\frac{1}{\nu^{2K+1}} + 2\|\Pi_{\infty(K)} \bar{w}^3\|_\nu + \|\Pi_{\infty(K)} \bar{w}^3 (L_{\bar{a}} (\bar{v} - S \bar{v}))_i\|_\nu \Big), \\
&\hspace{1.65cm} \sqrt{3} \max_{i=1,2,3} \|\Pi_{\infty(K)} \bar{w}^3 (L_{\bar{a}} (\bar{u} - S \bar{u}))_i\|_\nu, \\
&\hspace{1.65cm} \|[2 \bar{w} \big\langle \bar{u} - S \bar{u}, L_{\bar{a}} (\bar{u} - S \bar{u}) \big\rangle](0, \Omega)\|_X \frac{1}{\nu^{2K + 1}} \\
&\hspace{2.5cm} + \|\Pi_{\infty(K)} 3 w^2 \big\langle u - S u, L_{\bar{a}} (v - S v) \big\rangle\|_\nu + \sum_{i=1}^3 \|\Pi_{\infty(K)} (3\bar{w}^2 (\bar{u}_i - S \bar{u}_i))\|_\nu + \|\Pi_{\infty(K)} 3 \bar{w}^2 (\bar{u}_i - S^2 \bar{u}_i)\|_\nu \Big) \\
&\hspace{0.5cm} + \frac{1}{K+1} \max\Big( 
\max_{i=1,2,3} \Big(\|\Omega^2\|_X (\delta_{1i} + \delta_{2i}) + \sqrt{3}(2\|\bar{w}^3\|_\nu + \|\bar{w}^3 ( L_{\bar{a}} (\bar{v} - S \bar{v}) )_i\|_\nu)\Big), \\
&\hspace{2.56cm} 1 + \max_{i=1,2,3} \Big(\|2\Omega\|_X (\delta_{1i} + \delta_{2i}) + \sqrt{3} \|\bar{w}^3 ( L_{\bar{a}} (\bar{u} - S \bar{u}) )_i\|_\nu\Big), \\
&\hspace{2.56cm} \|3\bar{w}^2 \big\langle \bar{u} - S \bar{u}, L_{\bar{a}} (\bar{v} - S \bar{v}) \big\rangle\|_\nu + \sum_{i=1}^3 \|3\bar{w}^2 (\bar{u}_i - S \bar{u}_i)\|_\nu + \|3\bar{w}^2 (\bar{u}_i - S^2 \bar{u}_i)\|_\nu
\Big).
\end{align*}
}
Since the above expressions only involve polynomials, as well as the product of finite matrices, we can find an upper bound for $\| I - A DF(\bar{x}) \|_{\mathscr{B}(\cX_\nu, \cX_\nu)}$ by a finite number of computations, which are carried out by a computer with interval arithmetic.
\end{proof}

\begin{lemma}\label{lem:Z2}
\[
\sup_{x \in \textnormal{cl}(B_r (\bar{x}))} \| A D^2 F (x) \|_{\mathscr{B}(\cX_\nu, \mathscr{B}(\cX_\nu, \cX_\nu))} \le 10^4.
\]
\end{lemma}

\begin{proof}
Note that
\[
\sup_{x \in \textnormal{cl}(B_r (\bar{x}))} \| A D^2 F(x) \|_{\mathscr{B}(X, \mathscr{B}(\cX_\nu, \cX_\nu))}
\le \| A \|_{\mathscr{B}(\cX_\nu, \cX_\nu)} \sup_{x \in \textnormal{cl}(B_r (\bar{x}))} \| D^2 F(x) \|_{\mathscr{B}(\cX_\nu, \mathscr{B}(\cX_\nu, \cX_\nu))},
\]
with $\| A \|_{\mathscr{B}(\cX_\nu, \cX_\nu)} \le \max\left( \| A_\textnormal{finite} \|_{\mathscr{B}(\cX_\nu, \cX_\nu)},  \frac{1}{K+1} \right)$.
Moreover, observe that, for an operator $\mathfrak{T} \in C^2 (\ell^1)$,
\begin{align*}
\| D^2 \mathfrak{T}(x) \|_{\mathscr{B}(\ell^1, \mathscr{B}(\ell^1, \ell^1))}
&= \sup_{h \in \ell^1} \| D^2 \mathfrak{T}(x) h \|_{\mathscr{B}(\ell^1, \ell^1)} \\
&= \sup_{h \in \ell^1} \max_{j \in \N} \sum_{i \in \N} | \partial_j \sum_{l \in \N} \partial_l \mathfrak{T}_i (x) h_l | \\
&\le \max_{j \in \N} \sum_{i \in \N} \sup_{h \in \ell^1} | \partial_j \sum_{l \in \N} \partial_l \mathfrak{T}_i (x) h_l | \\
&= \max_{j \in \N} \sum_{i \in \N} \| \partial_j \nabla \mathfrak{T}_i (x) \|_ {\ell^\infty} \\
&= \max_{j \in \N} \sum_{i \in \N} \max_{l \in \N} | \partial_j \partial_l \mathfrak{T}_i (x) |.
\end{align*}
Since the norm $\| \, \cdot \, \|_{\cX_\nu}$ amounts to a weighted $\ell^1$ space, it is straightforward to obtain an analogous formula for $\| D^2 F(x) \|_{\mathscr{B}(\cX_\nu, \mathscr{B}(\cX_\nu, \cX_\nu))}$ of the form
\begin{align*}
&\sup_{x \in \textnormal{cl}(B_r (\bar{x}))} \| D^2 F(x) \|_{\mathscr{B}(\cX_\nu, \mathscr{B}(\cX_\nu, \cX_\nu))} \\
&\le \sup_{x \in \textnormal{cl}(B_r (\bar{x}))} \max\Bigg( \\
&\max\Big(
%\| \partial_a^2 \gamma \|_{\mathscr{B}(X, \mathscr{B}(X, X))},
\max_{i=1, 2, 3} \| \partial_{u_i} \partial_a \gamma \|_{\mathscr{B}(\mathcal{C}_\nu, \mathscr{B}(X, X))}, \| \partial_w \partial_a \gamma \|_{\mathscr{B}(\mathcal{C}_\nu, \mathscr{B}(X, X))}
\Big) \\
&\quad+ \max\Big(
%\| \partial_a^2 h \|_{\mathscr{B}(X, \mathscr{B}(X, \mathcal{C}_\nu))},
%\| \partial_\alpha \partial_a h \|_{\mathscr{B}(X, \mathscr{B}(X, \mathcal{C}_\nu))}, \\
%&\quad\qquad\qquad
\max_{i=1,2,3} \| \partial_{u_i} \partial_a h \|_{\mathscr{B}(X, \mathscr{B}(\mathcal{C}_\nu, \mathcal{C}_\nu))},
\max_{i=1,2,3} \| \partial_{v_i} \partial_a h \|_{\mathscr{B}(X, \mathscr{B}(\mathcal{C}_\nu, \mathcal{C}_\nu))},% \\
%&\quad\qquad\qquad
\| \partial_w \partial_a h \|_{\mathscr{B}(X, \mathscr{B}(\mathcal{C}_\nu, \mathcal{C}_\nu))}
\Big)
, \\
&\max_{i=1,2,3} \Big\{ \max\Big(
\| \partial_a \partial_{u_i} \gamma \|_{\mathscr{B}(X, \mathscr{B}(\mathcal{C}_\nu, X))},
\max_{j=1,2,3}\| \partial_{u_j} \partial_{u_i} \gamma \|_{\mathscr{B}(\mathcal{C}_\nu, \mathscr{B}(\mathcal{C}_\nu, X))},
\| \partial_w \partial_{u_i} \gamma \|_{\mathscr{B}(\mathcal{C}_\nu, \mathscr{B}(\mathcal{C}_\nu, X))}
\Big) \\
&\quad+% \max\Big(
%\| \partial_\beta \partial_{u_i} g_i \|_{\mathscr{B}(X, \mathscr{B}(\mathcal{C}_\nu, \mathcal{C}_\nu))},
%\max_{j=1,2,3} \| \partial_{u_j} \partial_{u_i} g_i \|_{\mathscr{B}(\mathcal{C}_\nu, \mathscr{B}(\mathcal{C}_\nu, \mathcal{C}_\nu))},% \\
%&\quad\qquad\qquad
%\max_{j=1,2,3} \| \partial_{v_j} \partial_{u_i} g_i \|_{\mathscr{B}(\mathcal{C}_\nu, \mathscr{B}(\mathcal{C}_\nu, \mathcal{C}_\nu))},
\| \partial_w \partial_{u_i} g_i \|_{\mathscr{B}(\mathcal{C}_\nu, \mathscr{B}(\mathcal{C}_\nu, \mathcal{C}_\nu))}
%\Big)
\\
&\quad+ \max\Big(
\| \partial_a \partial_{u_i} h \|_{\mathscr{B}(X, \mathscr{B}(\mathcal{C}_\nu, \mathcal{C}_\nu))},
%\| \partial_\alpha \partial_{u_i} h \|_{\mathscr{B}(X, \mathscr{B}(\mathcal{C}_\nu, \mathcal{C}_\nu))}, \\
%&\quad\qquad\qquad
%\max_{j=1,2,3}\| \partial_{u_j} \partial_{u_i} h\|_{\mathscr{B}(\mathcal{C}_\nu, \mathscr{B}(\mathcal{C}_\nu, \mathcal{C}_\nu))},
\max_{j=1,2,3}\| \partial_{v_j} \partial_{u_i} h  \|_{\mathscr{B}(\mathcal{C}_\nu, \mathscr{B}(\mathcal{C}_\nu, \mathcal{C}_\nu))},% \\
%&\quad\qquad\qquad
\| \partial_w \partial_{u_i} h \|_{\mathscr{B}(\mathcal{C}_\nu, \mathscr{B}(\mathcal{C}_\nu, \mathcal{C}_\nu))}
\Big)
\Big\}, \\
&\max_{i=1,2,3} \Big\{
\| \partial_a \partial_{v_i} h \|_{\mathscr{B}(X, \mathscr{B}(\mathcal{C}_\nu, \mathcal{C}_\nu))},
%\| \partial_\alpha \partial_{v_i} h \|_{\mathscr{B}(X, \mathscr{B}(\mathcal{C}_\nu, \mathcal{C}_\nu))}, \\
%&\quad\qquad\qquad
\max_{j=1,2,3} \| \partial_{u_j} \partial_{v_i} h \|_{\mathscr{B}(\mathcal{C}_\nu, \mathscr{B}(\mathcal{C}_\nu, \mathcal{C}_\nu))},
%\max_{j=1,2,3} \| \partial_{v_j} \partial_{v_i} h \|_{\mathscr{B}(\mathcal{C}_\nu, \mathscr{B}(\mathcal{C}_\nu, \mathcal{C}_\nu))}, \\
%&\quad\qquad\qquad
\| \partial_w \partial_{v_i} h \|_{\mathscr{B}(\mathcal{C}_\nu, \mathscr{B}(\mathcal{C}_\nu, \mathcal{C}_\nu))}
\Big\}, \\
&\max\Big(
\| \partial_a \partial_w \gamma \|_{\mathscr{B}(X, \mathscr{B}(\mathcal{C}_\nu, X))},
\max_{i=1,2,3} \| \partial_{u_i} \partial_w \gamma \|_{\mathscr{B}(\mathcal{C}_\nu, \mathscr{B}(\mathcal{C}_\nu, X))},
\| \partial_w^2 \gamma \|_{\mathscr{B}(\mathcal{C}_\nu, \mathscr{B}(\mathcal{C}_\nu, X))}
\Big) \\
&\quad+ \sum_{i=1}^3 \max\Big(
%\| \partial_\beta \partial_w g_i \|_{\mathscr{B}(X, \mathscr{B}(\mathcal{C}_\nu, \mathcal{C}_\nu))},
\max_{j=1,2,3} \| \partial_{u_j} \partial_w g_i \|_{\mathscr{B}(\mathcal{C}_\nu, \mathscr{B}(\mathcal{C}_\nu, \mathcal{C}_\nu))},% \\
%&\quad\qquad\qquad\max_{j=1,2,3} \| \partial_{v_j} \partial_w g_i \|_{\mathscr{B}(\mathcal{C}_\nu, \mathscr{B}(\mathcal{C}_\nu, \mathcal{C}_\nu))},
\| \partial_w^2 g_i \|_{\mathscr{B}(\mathcal{C}_\nu, \mathscr{B}(\mathcal{C}_\nu, \mathcal{C}_\nu))}
\Big) \\
&\quad+ \max\Big(
\| \partial_a \partial_w h \|_{X, \mathscr{B}(\mathcal{C}_\nu, \mathcal{C}_\nu))},
%\| \partial_\alpha \partial_w h \|_{X, \mathscr{B}(\mathcal{C}_\nu, \mathcal{C}_\nu))}, \\
%&\quad\qquad\qquad
\max_{i=1,2,3} \| \partial_{u_i} \partial_w h \|_{\mathscr{B}(\mathcal{C}_\nu, \mathscr{B}(\mathcal{C}_\nu, \mathcal{C}_\nu))},
\max_{i=1,2,3} \| \partial_{v_i} \partial_w h \|_{\mathscr{B}(\mathcal{C}_\nu, \mathscr{B}(\mathcal{C}_\nu, \mathcal{C}_\nu))}, \\
&\quad\qquad\qquad\| \partial_w^2 h \|_{\mathscr{B}(\mathcal{C}_\nu, \mathscr{B}(\mathcal{C}_\nu, \mathcal{C}_\nu))}
\Big)
\Bigg) \\
&\le \max\Big(
6 \max\big( \hat{w}^2 \hat{u}_3, \hat{w} \hat{u}_3^2 \big) + 3\max\big(\hat{w}^3 \hat{v}_3, \hat{w}^3 \hat{u}_3, 3 \hat{w}^2 \hat{u}_3 \hat{v}_3\big), \\
&\hspace{1.5cm} \max_{i=1,2,3} \Big\{6 \max\big(\delta_{3i} \hat{w}^2 \hat{u}_i, (1 + \delta_{3i}(\hat{a}-1)) \hat{w}^2, 2 (1 + \delta_{3i}(\hat{a}-1)) \hat{w} \hat{u}_i\big) \\
&\hspace{2.75cm} + 6 \sqrt{3} \hat{w}^2 \\
&\hspace{2.75cm} + 3 \max\big(\delta_{3i} \hat{w}^3 \hat{v}_i, (1 + \delta_{3i}(\hat{a}-1)) \hat{w}^3, 3 (1 + \delta_{3i}(\hat{a}-1)) \hat{w}^2 \hat{v}_i\big)\Big\}, \\
&\hspace{1.5cm} \max_{i=1,2,3} \Big\{ 3 \max\big(\delta_{3i} \hat{w}^3 \hat{u}_i, (1 + \delta_{3i}(\hat{a}-1)) \hat{w}^3, 3(1+\delta_{3i}(\hat{a}-1)) \hat{w}^2 \hat{u}_i\big) \Big\}, \\
&\hspace{1.5cm} 6 \max \big(\hat{w} \hat{u}_3^2, 2 \hat{w} \hat{u}_1, 2\hat{w} \hat{u}_2, 2\hat{w} \hat{a} \hat{u}_3, \hat{u}_1^2 + \hat{u}_2^2 + \hat{a} \hat{u}_3^2\big) \\
&\hspace{2.75cm} + 6 \sqrt{3} \sum_{i=1}^3 \max(\hat{w}^2, 2 \hat{w} \hat{u}_i) \\
&\hspace{2.75cm} + 9 \max\big(\hat{w}^2 \hat{u}_3 \hat{v}_3, \hat{w}^2 \hat{v}_1, \hat{w}^2 \hat{v}_2, \hat{w}^2 \hat{a} \hat{v}_3,
\hat{w}^2 \hat{u}_1, \hat{w}^2 \hat{u}_2, \hat{w}^2 \hat{a} \hat{u}_3,
2\hat{w} (\hat{u}_1 \hat{v}_1 + \hat{u}_2 \hat{v}_2 + \hat{a} \hat{u}_3 \hat{v}_3)\big)
\Bigg),
\end{align*}
where we used the Banach algebra property, and $\hat{a} \bydef \|\bar{a}\|_X + r$, $\hat{u}_i \bydef \|\bar{u}_i\|_\nu + r$, $\hat{v}_i \bydef \|\bar{v}_i\|_\nu + r$, $\hat{w} \bydef \|\bar{w}\|_\nu + r$.

Therefore, we can find an upper bound for $\sup_{x \in \textnormal{cl}(B_r (\bar{x}))} \| A D^2 F (x) \|_{\mathscr{B}(\cX_\nu, \mathscr{B}(\cX_\nu, \cX_\nu))}$ by a finite number of computations, which are carried out by a computer with interval arithmetic.
\end{proof}

Lemma \ref{lem:Y}, \ref{lem:Z1} and \ref{lem:Z2} yield the desired contraction of $T$.

\begin{theorem}\label{thm:contraction}
\[
T : \textnormal{cl}(B_r (\bar{x})) \to \textnormal{cl}(B_r (\bar{x})),
\]
and, for all $x, x' \in \textnormal{cl}(B_r (\bar{x}))$,
\[
\| T (x) - T (x') \|_{\cX_\nu} \le \kappa \| x - x' \|_{\cX_\nu},
\]
where $\kappa =  \frac{91}{100} < 1$.
\end{theorem}

\begin{proof}
Let $x, x' \in \textnormal{cl}(B_r (\bar{x}))$.
Taylor's theorem yields
\begin{align*}
\|T (x) - \bar{x}\|_{\cX_\nu}
&= \| T (\bar{x}) - \bar{x} + [DT(\bar{x})](x - \bar{x}) + \int_0^1 [DT(\bar{x} + t(x - \bar{x})) - DT(\bar{x})] (x - \bar{x}) \, dt \|_{\cX_\nu}
\\
&\le \|A F(\bar{x})\|_{\cX_\nu} + \|I - A DF(\bar{x})\|_{\mathscr{B}(\cX_\nu, \cX_\nu)} r + \frac{1}{2} \sup_{x'' \in \textnormal{cl} (B_r (\bar{x}))} \|A D^2 F(x'')\|_{\mathscr{B}(\cX_\nu, \mathscr{B}(\cX_\nu, \cX_\nu))} r^2
\\
&\le 10^{-8} + \frac{9}{10} r + \frac{10^4}{2} r^2
\\
&\le r,
\end{align*}
and, by the mean value theorem, we have
\begin{align*}
\|T (x) - T (x')\|_{\cX_\nu}
&\le \sup_{x'' \in \textnormal{cl} (B_r (\bar{x}))} \| D T (\bar{x}) - D T (\bar{x}) + D T (x'') \|_{\mathscr{B}(\cX_\nu, \cX_\nu)} \|x - x'\|_{\cX_\nu} \\
&\le \left( \|I - A DF(\bar{x})\|_{\mathscr{B}(\cX_\nu, \cX_\nu)} + \sup_{x'' \in \textnormal{cl} (B_r (\bar{x}))} \|A D^2 F(x'')\|_{\mathscr{B}(\cX_\nu, \mathscr{B}(\cX_\nu, \cX_\nu))} r \right) \| x - x' \|_{\cX_\nu} \\
&\le \left(\frac{9}{10} + 10^4 \times 10^{-6}\right) \|x - x'\|_{\cX_\nu} \\
& = \kappa \|x - x'\|_{\cX_\nu}.
\end{align*}
Therefore, $T$ satisfies the Banach Fixed-Point Theorem in $\textnormal{cl} (B_r (\bar{x}))$.
\end{proof}

At last, we have the existence of a family, parameterized by $\Omega \in [0, 1]$, of relative choreographies.

\begin{corollary}
There exists a unique zero of $F$ in $\textnormal{cl}(B_r (\bar{x}))$.
\end{corollary}

\begin{proof}
By Theorem \ref{thm:contraction}, $\|DT(\bar{x})\|_{\mathscr{B}(\cX_\nu, \cX_\nu)} = \| I - A D F(\bar{x}) \|_{\mathscr{B}(\cX_\nu, \cX_\nu)} < 1$ which implies that $A D F(\bar{x})$ is invertible.
Hence, $A$ is surjective.
To show injectivity, note that $A_\textnormal{tail} : \Pi_{\infty(K)} \cX_\nu \to \Pi_{\infty(K)} \cX_\nu$ is injective.
Moreover, $A_\textnormal{finite} : \Pi_K \cX_\nu \to \Pi_K \cX_\nu$ is by construction equivalent to a finite dimensional square matrix which implies that surjectivity is equivalent to injectivity.
Thus, by injectivity of $A$, we have that the fixed-point of $T$ is a zero of $F$.
\end{proof}

%!TEX root = marchal.tex

\section{Proof of Marchal's Conjecture}
\label{sec:aposteriori}

In this section, we prove that the branch of solutions of the functional equation \eqref{eq:zero-finding-problem} proven to exist in Section \ref{sec:proof} corresponds to the $P_{12}$ Marchal's family of relative periodic choreographies and that this family ends in the figure eight of Chenciner and Montgomery.
This concludes the proof of Marchal's conjecture.
Before doing so, let us consider the following auxiliary result giving sufficient conditions for an isolated zero of a map to satisfy an invariance property.

\begin{lemma}\label{lem:aposteriori}
Let $\mathfrak{F} : \mathcal{D} (\mathfrak{F}) \subset \mathfrak{X} \to \mathfrak{X}$ and $\mathfrak{S} \in \mathscr{B}(\mathfrak{X}, \mathfrak{X})$ such that $\mathfrak{F} \circ \mathfrak{S} = \mathfrak{S} \circ \mathfrak{F}$ and $\| \mathfrak{S} \|_{\mathscr{B}(\mathfrak{X}, \mathfrak{X})} \le 1$.
Consider $\bar{x}, \tilde{x} \in \mathcal{D} (\mathfrak{F})$ and $r > 0$ such that $\tilde{x}$ is the unique element in $\textnormal{cl}(B_r (\bar{x})) \cap \mathcal{D} (\mathfrak{F})$ satisfying $\mathfrak{F} (\tilde{x}) = 0$.
If $\mathfrak{S} \bar{x} = \bar{x}$, then $\mathfrak{S} \tilde{x} = \tilde{x}$.
\end{lemma}

\begin{proof}
The property $\mathfrak{F} \circ \mathfrak{S} = \mathfrak{S} \circ \mathfrak{F}$ implies that $\mathfrak{S} \tilde{x}$ is a zero of $\mathfrak{F}$.
From $\| \mathfrak{S} \|_{\mathscr{B}(\mathfrak{X}, \mathfrak{X})} \le 1$, it follows that $\mathfrak{S} \tilde{x} \in \textnormal{cl}(B_r (\bar{x}))$.
The fact that $\tilde{x}$ is the unique zero of $\mathfrak{F}$ in $\textnormal{cl}(B_r (\bar{x}))$ concludes the proof.
\end{proof}

We now finish the demonstration of our main result.
For the sake of clarity, we give a precise statement of Theorem \ref{thm:MarchalWasRight} as follows.

\begin{theorem}\label{thm:main}
The $3$-body problem with the equation \eqref{eq:N_body} has the family of relative periodic solutions $q = (q_1, q_2, q_3)$ of the form
\[
q_j (t) = e^{-\Omega \bar{J} t} U(t + 4\pi j/3, \Omega), \qquad j = 1, 2, 3,
\]
where $U(t, \Omega) : \mathbb{R} / 2\pi \mathbb{Z} \times [0, 1] \to \mathbb{R}^3$ is an analytic function obtained as the unique fixed-point of $T$ in Section \ref{sec:proof}.
At $\Omega = 1$, the solution is the Lagrange triangle.
At $\Omega = 0$, the solution is of the form $q = (0, q_y, q_z)$, and its orbit is an eight-shaped figure with exactly the same properties as the main theorem of \cite{ProofEight}, and reported in Lemma \ref{lem:eight_shape}.
\end{theorem}

\begin{proof}
From Theorem \ref{thm:contraction}, we have proven existence (and local uniqueness) of a zero $\tilde{x} = (\tilde{a}, \tilde{\beta}, \tilde{\alpha}, \tilde{u}, \tilde{v}, \tilde{w}) \in \mathcal{X}_\nu$ of $F$, defined in \eqref{eq:zero-finding-problem}.
The functions $\tilde{u}, \tilde{v}, \tilde{w}$ are analytic with respect to $t$ due to our choice of norm \eqref{eq:nu_mu_norm}.
As a matter of fact, since the fixed-point operator $T$, defined in \eqref{eq:fixed-point-operator}, is analytic with respect to the parameter $\Omega$, it follows from the uniform contraction that all the functions $\tilde{a}, \tilde{\beta}, \tilde{\alpha}, \tilde{u}, \tilde{v}, \tilde{w}$ are analytic.

Furthermore, according to Lemma \ref{lem:beta} and \ref{lem:alpha}, it follows that $\tilde{\alpha}(\Omega) = \tilde{\beta}(\Omega) = 0$ for all $\Omega \in [0, 1]$.
Therefore,
\[
U = L_{\sqrt{\tilde{a}}} \tilde{u}
\]
is an analytic periodic solution of \eqref{eq:N_body_delay}.

To show that $U(t, 1)$ is the Lagrange triangle, we note that $\tilde{x}|_{\Omega=1}$ is the unique zero of $F|_{\Omega=1}$ in $\textnormal{cl}(B_r (\bar{x}|_{\Omega=1}))$.
By imposing that $\bar{x}|_{\Omega=1}$ is the Lagrange triangle, it follows that $\tilde{x}|_{\Omega=1} = \bar{x}|_{\Omega=1}$.
Specifically, recalling Lemma \ref{lem:triangle}, we set
\[
\bar{x}|_{\Omega=1} = (0, 0, 0, \bar{u}(t, 1), \partial_t \bar{u}(t, 1), 3^{-1/3}),
\]
where
\[
\bar{u}(t, 1) =  \begin{pmatrix} 3^{-1/6} \cos 2 t \\ -3^{-1/6} \sin 2 t \\ \cos t \end{pmatrix}.
\]

Now, the Banach space $\cX_\nu$ contains complex-valued functions, hence we cannot a priori conclude that $U$ is the $P_{12}$ family.
To solve this, we define $\Sigma : \cX_\nu \to \cX_\nu$ by
\[
\Sigma x \bydef (a^*, \beta^*, \alpha^*, u^*, v^*, w^*), \qquad \text{for all } x = (a, \beta, \alpha, u, v, w) \in \cX_\nu,
\]
where, given a function $h : \R \to \C$ (resp. $h : \R^2 \to \C$), we use the notation $[h^*](\Omega) \bydef h(\Omega)^*$ (resp. $[h^*](t, \Omega) \bydef h(t, \Omega)^*$), with the superscript $^*$ representing the complex conjugate.
The functions $a, \beta, \alpha, u, v, w$ are real-valued if and only if $\Sigma x = x$.
Again, we constraint the Fourier-Chebyshev coefficients (which are finite in number) of the approximation of the $P_{12}$ family such that $\Sigma \bar{x} = \bar{x}$.
It is straightforward to verify that Lemma \ref{lem:aposteriori} is satisfied with $\mathfrak{F} = F$, $\mathfrak{S} = \Sigma$.
Therefore,
\[
\Sigma \tilde{x} = \tilde{x},
\]
which completes the proof that $U$ is real-valued and represents the $P_{12}$ family.

Lastly, we show that $U(t, 0)$ is the figure eight described in Lemma \ref{lem:eight_shape}.
First, let us prove that $U(t, 0)$ lies in the $yz$-plane for all $t$.
Let $\Sigma_0 : \cX_\nu \to \cX_\nu$ be given by
\[
\Sigma_0 x \bydef (a, \beta, \alpha, -u_1, u_2, u_3, -v_1, v_2, v_3, w),
\]
for all $x = (a, \beta, \alpha, u, v, w) \in \cX_\nu$, with $u = (u_1, u_2, u_3)$ and $v = (v_1, v_2, v_3)$.
The functions $u_1(\,\cdot\,, 0), v_1(\,\cdot\,, 0)$ are zero if and only if $\Sigma_0 x|_{\Omega=0} = x|_{\Omega=0}$.
As done for the Lagrange triangle, we constraint the finite number of Fourier-Chebyshev coefficients defining the approximation of the $P_{12}$ family such that $\Sigma_0 \bar{x} |_{\Omega=0} = \bar{x} |_{\Omega=0}$.
Note that $\tilde{x} |_{\Omega=0}$ is the unique zero of $F|_{\Omega = 0}$ in $\textnormal{cl}(B_r (\bar{x} |_{\Omega=0}))$.
Then, Lemma \ref{lem:aposteriori} holds with $\mathfrak{F} = F|_{\Omega = 0}$, $\mathfrak{S} = \Sigma_0$.
Whence,
\[
\Sigma_0 \tilde{x} |_{\Omega=0} = \tilde{x} |_{\Omega=0},
\]
meaning that $U(t, 0)$ lies in the $yz$-plane for all $t$ as desired.

To establish that the orbit of $U(t, 0)$ is the same as the one proven by Chenciner and Montgomery \cite{ProofEight}, we satisfy Lemma \ref{lem:eight_shape}.
Points $(i)$-$(ii)$-$(iii)$ follow immediately from the symmetries considered in $\cX_\nu$.
It remains to verify $(iv)$.
To do so, first observe that
\[
\mu(t) \bydef U(t, 0) \times \dot{U}(t, 0) =
\begin{pmatrix}
U_y (t, 0) \dot{U}_z (t, 0) - U_z (t, 0) \dot{U}_y (t, 0) \\ 0 \\ 0
\end{pmatrix}.
\]
From our symmetries, detailed in \eqref{eq:U_sym}, it is straightforward to prove that $\mu(\pi/2) = 0$.
Since the function $\mu$ is known with a margin of error related to $r =10^{-6}$ (see Theorem \ref{thm:contraction}), we cannot directly determine its sign close to $\pi/2$.
In fact, it holds that $\dot{\mu}(\pi/2) = \ddot{\mu}(\pi/2) = 0$.
Thus, using the computer and interval arithmetic, we perform the following steps:
\begin{enumerate}
\item we show that $\mu(t) < 0$ on $[0, 3/2]$. It remains to study the sign of $\mu$ on $[3/2, \pi/2]$.
\item we show that $\dddot{\mu}(t) > 0$ on $[3/2, \pi/2]$.
Note that we are able to control derivatives since $U \in \mathcal{U}_\nu$ with $\nu > 1$, indeed the Fourier coefficients $a_k$ of a function in this space enjoy an exponential decay rate $|a_k| \le (\sum_{k \in \Z} |a_k|) \nu^{-|k|}$.
Hence, $\ddot{\mu}$ is strictly monotone on $[3/2, \pi/2]$, therefore $\ddot{\mu}(t) = 0$ if and only if $t = \pi/2$ which means that $\ddot{\mu}$ has a sign on $[3/2, \pi/2)$.
Repeating this argument, it follows that $\dot{\mu}$ has a sign on $[3/2, \pi/2)$, and in turns $\mu(t) < 0$ for all $t \in [3/2, \pi/2)$ as desired.
\end{enumerate}
This concludes the proof that $U$ is the $P_{12}$ family, starting at the Lagrange triangle and ending at the figure eight.
\end{proof}

\section{Future work} \label{sec:future}

By way of conclusions, we consider several additional directions which we view as natural continuations of the present work.

First, the figure eight choreography is a planar curve having three obvious symmetry axes: two of them in the plane of the eight, and a third one in the direction orthogonal to the plane.
In our setup, $q_z$ corresponds to the direction along which the eight is the longest, $q_y$ the orthogonal direction to $q_z$ in the plane that contains the eight, and $q_x$ orthogonal to this plane.
In \cite{MR2134901}, Chenciner, Féjoz and Montgomery find three families of relative choreographies that bifurcate out of the eight corresponding to each of the symmetry axes, named $\Gamma_1$, $\Gamma_2$ and $\Gamma_3$ respectively.

As stated in \cite{MR2134901}, $\Gamma_1$ corresponds to the $P_{12}$ family studied here.
The family $\Gamma_2$, on the other hand, is related to rotations about the $q_y$ direction, forming an out-of-plane choreographic family bifurcating from the eight and having fewer symmetries than $P_{12}$.
Numerical evidence from \cite{Arioli2006,Barutello_2004,MarchalI} indicates that the family of choreographic orbits -- which branches off $P_{12}$ and has axial symmetry in the rotating frame -- relates to $\Gamma_2$.
Numerical evidence also suggests that this axial family at $\Omega = 3/2$ draws the figure eight again in the inertial frame (now located in the plane $xOz$, where $q_x$ is the direction along which the eight is the longest).
It would be interested to extend the proof done here, to have a different mechanism to get from the Lagrange triangle to the figure eight using the connection of $P_{12}$ with the axial family.

The family $\Gamma_3$ was already known numerically to Michel Hénon using computations similar to \cite{Henon1974} and has been considered in other works, see for example \cite{Calleja2018}.
Numerical evidence suggests that there is no off-plane bifurcation from $\Gamma_3$ and making this statement rigorous is also an interesting future direction.

A numerical study of linear stability is reported in \cite{MR2134901}.
Chenciner, Féjoz and Montgomery mention that the $P_{12}$ family becomes unstable as soon as the family leaves the eight: so, we do not expect to have an interval of stability. However \cite{MR2134901} and  \cite{MarchalI} suggest an interval of stability in $\Gamma_2$ and $\Gamma_3$ from the eight.

Finally, we remark that in the paper \cite{MarchalI}, four of the authors of the present work stated the following generalization of Marchal's conjecture:

\begin{conjecture}\label{conj:first}
For any odd number of bodies, the $n$-gon choreography and the $n$-body figure eight are in the same continuation class.
\end{conjecture}

After \cite{MarchalI} was published, Chenciner stressed, in private communication, that the conjecture of Marchal was not just about having a connection between the triangle and the eight.
It was also about showing that both choreographies actually belong to the $P_{12}$ family.
In this manner, Conjecture \ref{conj:first} should not be called a generalization, but a conjecture in its own right, supported by the numerical evidence in \cite{MarchalI}.

Specifically, Theorem 38 in \cite{Gar-Ize-13} states that for any $n \ge 3$ and for each integer $k$ such that $1 \le k \le n/2$, the $n$-polygonal relative equilibrium has a least one family of spatial periodic solutions with prescribed symmetries in terms of $k$.
In \cite{MarchalI} the numerical computations support the conjecture that the $n$-gon choreography and the figure eight choreography are in the same continuation class through the family with symmetries $k=2$.
On the other hand, numerical computations in \cite{MR2480953} support the conjecture that the $n$-gon choreography and the figure eight are in the same continuation class through the generalized $P_{12}$ family, which corresponds in \cite{Gar-Ize-13} to the family with symmetries $k=(n-1)/2$.

Therefore, the generalized Marchal's conjecture for odd number of bodies should read as follows.

\begin{conjecture}
\textit{(Generalized Marchal's Conjecture - Updated)}
For any odd number of bodies, the $n$-gon choreography and the figure eight are in the same continuation class through two different families of relative choreographies.
\end{conjecture}

This conjecture, which is supported by the available numerical evidence, requires settling many open problems.
For example the existence and the uniqueness of the figure eight for any odd number of bodies.
This uniqueness is still open for $n=3$.
Actually, only for $n=3$ do the two different families start with the same symmetries.
Furthermore, the numerical evidence in \cite{MarchalI} shows that the second family that connects the triangle with the eight arises precisely as a symmetry-breaking from the $P_{12}$ family.
This fact does not contradict the local uniqueness obtained in this paper because the second family is precisely a symmetry-breaking outside the space of discrete symmetries $\mathcal{U}$ considered in Section \ref{sec:setUp}.

The verification of the \emph{generalized Marchal's conjecture} for $n=3$, and also for other odd number of bodies, may be amenable to the methods developed in the present paper.
We also refer the reader to the works \cite{Yu2017SimpleChoreographies,Yu2021ConnectingPlanarChains} for a complementary perspective on the problem, based on action minimization under symmetry and topological constraints.
% The authors plan to address these questions in forthcoming works.

\section{Acknowledgments}

We would like to thank A. Chenciner, J. F\'{e}joz, R. Montgomery and C. Sim\'{o} for fruitful discussions.
R.~Calleja was partially supported by UNAM-PAPIIT under grants No. IN 103423 and IN 104725.
O.~H\'{e}not was supported by ANR under project No. CAPPS: ANR-23-CE40-0004-01, and by NSTC under grant No. 115-2115-M-002-001-MY2.
J.-P.~Lessard was supported by NSERC.
J.~D.~Mireles~James was partially supported by NSF under grant No. DMS 2307987.

%%%%%%%%%%%%%%%%%%%%%%%%%%

\bibliographystyle{abbrv}
\bibliography{references}

@article{Yu2017SimpleChoreographies,
  author  = {Yu, Guowei},
  title   = {Simple Choreographies of the Planar {Newtonian} {$N$}-body Problem},
  journal = {Archive for Rational Mechanics and Analysis},
  volume  = {225},
  number  = {2},
  pages   = {901--935},
  year    = {2017},
  doi     = {10.1007/s00205-017-1116-1},
  url     = {https://link.springer.com/article/10.1007/s00205-017-1116-1},
}

@article{Yu2021ConnectingPlanarChains,
  author  = {Yu, Guowei},
  title   = {Connecting Planar Linear Chains in the Spatial {$N$}-body Problem},
  journal = {Annales de l'Institut Henri Poincar\'e Analyse non lin\'eaire},
  volume  = {38},
  number  = {4},
  pages   = {1115--1144},
  year    = {2021},
  doi     = {10.1016/j.anihpc.2020.10.004},
  url     = {https://www.numdam.org/articles/10.1016/j.anihpc.2020.10.004/},
}

@phdthesis{minton2013computer,
  author       = {Gregory T. Minton},
  title        = {Computer-Assisted Proofs in Geometry and Physics},
  school       = {Massachusetts Institute of Technology},
  type         = {{Ph.D.} dissertation},
  year         = {2013},
  address      = {Cambridge, MA, USA},
  department   = {Department of Mathematics},
  url          = {https://dspace.mit.edu/handle/1721.1/84405},
}

@article{MR2679365,
	author = {Arioli, Gianni and Koch, Hans},
	doi = {10.1007/s00205-010-0309-7},
	journal = {Archive for Rational Mechanics and Analysis},
	number = {3},
	pages = {1033--1051},
	title = {Computer-assisted methods for the study of stationary solutions in dissipative systems, applied to the {K}uramoto-{S}ivashinski equation},
	url = {https://doi.org/10.1007/s00205-010-0309-7},
	volume = {197},
	year = {2010}}

@book{craig2008hamiltonian,
  title     = {Hamiltonian Dynamical Systems and Applications},
  editor    = {Craig, Walter},
  year      = {2008},
  publisher = {Springer Dordrecht},
  series    = {NATO Science for Peace and Security Series B: Physics and Biophysics},
  doi       = {10.1007/978-1-4020-6964-2},
  isbn      = {978-1-4020-6964-2}
}

@article{Arioli2006,
	author = {Arioli, Gianni and Barutello, Vivina and Terracini, Susanna},
	doi = {10.1007/s00220-006-0111-4},
	journal = {Communications in Mathematical Physics},
	pages = {439--463},
	title = {A New Branch of Mountain Pass Solutions for the Choreographical $3$-Body Problem},
	volume = {268},
	year = {2006}}

@article{Barutello_2004,
	author = {Barutello, Vivina and Terracini, Susanna},
	doi = {10.1088/0951-7715/17/6/002},
	journal = {Nonlinearity},
	pages = {2015},
	title = {Action minimizing orbits in the $n$-body problem with simple choreography constraint},
	volume = {17},
	year = {2004}}

@incollection{Bucker2006,
	author = {B\"{u}cker, H. Martin and Corliss, George F.},
	booktitle = {Automatic Differentiation: Applications, Theory, and Implementations},
	doi = {10.1007/3-540-28438-9_28},
	pages = {321--322},
	publisher = {Springer Berlin},
	series = {Lecture Notes in Computational Science and Engineering},
	title = {A Bibliography of Automatic Differentiation},
	volume = {50},
	year = {2006}}

@misc{IntervalArithmeticJulia,
	author = {Benet, Luis and Sanders, David P.},
	howpublished = {\url{https://github.com/JuliaIntervals/IntervalArithmetic.jl}},
	note = {Software},
	title = {{I}nterval{A}rithmetic.jl},
	year = {2022}}

@article{10.1137/20M1343464,
	author = {van den Berg, Jan Bouwe and Lessard, Jean-Philippe and Queirolo, Elena},
	doi = {10.1137/20M1343464},
	journal = {SIAM Journal on Applied Dynamical Systems},
	title = {Rigorous verification of {H}opf bifurcations via desingularization and continuation},
	volume = {20},
	year = {2021}}

@article{MR3444942,
	author = {van den Berg, Jan Bouwe and Lessard, Jean-Philippe},
	doi = {10.1090/noti1276},
	journal = {Notices of the American Mathematical Society},
	number = {9},
	pages = {1057--1061},
	title = {Rigorous numerics in dynamics},
	volume = {62},
	year = {2015}}

@book{MR3822720,
	author = {van den Berg, Jan Bouwe and Lessard, Jean-Philippe},
	publisher = {American Mathematical Society},
	series = {Proceedings of Symposia in Applied Mathematics Series},
	title = {Rigorous Numerics in Dynamics: AMS Short Course},
	year = {2018}}

@unpublished{OlivierBreden,
	author = {Breden, Maxime and H\'{e}not, Olivier},
	note = {In preparation},
	title = {Efficient rigorous continuation via {C}hebyshev series expansion},
	year = {2024}}

@article{Breden,
	author = {Breden, Maxime},
	doi = {10.1137/22M1493197},
	journal = {SIAM Journal on Applied Dynamical Systems},
	number = {2},
	pages = {765-801},
	title = {A Posteriori Validation of Generalized Polynomial Chaos Expansions},
	volume = {22},
	year = {2023}}

@article{Julia,
	author = {Bezanson, Jeff and Edelman, Alan and Karpinski, Stefan and Shah, Viral B.},
	doi = {10.1137/141000671},
	journal = {SIAM Review},
	number = {1},
	pages = {65--98},
	title = {Julia: a fresh approach to numerical computing},
	volume = {59},
	year = {2017}}

@incollection{MR2531684,
	author = {Charpentier, Isabelle and Lejeune, Arnaud and Potier-Ferry, Michel},
	booktitle = {Advances in Automatic Differentiation},
	doi = {10.1007/978-3-540-68942-3\_13},
	pages = {139--149},
	publisher = {Springer, Berlin},
	series = {Lecture Notes in Computational Science and Engineering},
	title = {The {D}iamant approach for an efficient automatic differentiation of the asymptotic numerical method},
	volume = {64},
	year = {2008}}

@article{MarchalI,
	author = {Calleja, Renato and Garc\'{i}a-Azpeitia, Carlos and Lessard, Jean-Philippe and Mireles James, Jason D.},
	doi = {10.1007/s10569-021-10009-9},
	journal = {Celestial Mechanics and Dynamical Astronomy},
	number = {10},
	title = {From the {L}agrange polygon to the figure eight {I}: numerical evidence extending a conjecture of {M}archal},
	volume = {133},
	year = {2021}}

@article{Calleja2018,
	author = {Calleja, Renato and Doedel, Eusebius and Garc\'{i}a-Azpeitia, Carlos},
	doi = {10.1007/s10569-018-9841-9},
	journal = {Celestial Mechanics and Dynamical Astronomy},
	title = {Symmetries and choreographies in families bifurcating from the polygonal relative equilibrium of the $n$-body problem},
	volume = {130},
	year = {2018}}

@article{MR4208440,
	author = {Calleja, Renato and Garc\'{i}a-Azpeitia, Carlos and Lessard, Jean-Philippe and Mireles James, J. D.},
	doi = {10.1088/1361-6544/abcb08},
	journal = {Nonlinearity},
	number = {1},
	pages = {313--349},
	title = {Torus knot choreographies in the $n$-body problem},
	volume = {34},
	year = {2021}}

@article{10.1007/s10884-023-10279-x,
	author = {Church, Kevin E. M. and Queirolo, Elena},
	doi = {10.1007/s10884-023-10279-x},
	journal = {Journal of Dynamics and Differential Equations},
	title = {Computer-assisted proofs of {H}opf bubbles and degenerate {H}opf bifurcations},
	year = {2023}}

@article{ProofEight,
	author = {Chenciner, Alain and Montgomery, Richard},
	doi = {10.2307/2661357},
	journal = {Annals of mathematics},
	pages = {881--901},
	title = {A Remarkable Periodic Solution of the Three-Body Problem in the Case of Equal Masses},
	volume = {152},
	year = {2000}}

@article{MR2134901,
	author = {Chenciner, Alain and F\'{e}joz, Jacques and Montgomery, Richard},
	doi = {10.1088/0951-7715/18/3/024},
	journal = {Nonlinearity},
	number = {3},
	pages = {1407--1424},
	title = {Rotating eights {I}: the three {$\Gamma_i$} families},
	volume = {18},
	year = {2005}}

@article{MR2480953,
	author = {Chenciner, Alain and F\'{e}joz, Jacques},
	doi = {10.1134/S1560354709010079},
	journal = {Regular and Chaotic Dynamics},
	number = {1},
	pages = {64--115},
	title = {Unchained polygons and the {$N$}-body problem},
	volume = {14},
	year = {2009}}

@incollection{MR1919833,
	author = {Chenciner, Alain and Gerver, Joseph and Montgomery, Richard and Sim\'{o}, Carles},
	booktitle = {Geometry, mechanics, and dynamics},
	doi = {10.1007/0-387-21791-6\_9},
	pages = {287--308},
	publisher = {Springer, New York},
	title = {Simple choreographic motions of {$N$} bodies: a preliminary study},
	year = {2002}}

@article{MR2425050,
	author = {Chenciner, Alain and F\'{e}joz, Jacques},
	doi = {10.3934/dcdsb.2008.10.421},
	journal = {Discrete and Continuous Dynamical Systems Series B},
	number = {2-3},
	pages = {421--438},
	title = {The flow of the equal-mass spatial 3-body problem in the neighborhood of the equilateral relative equilibrium},
	volume = {10},
	year = {2008}}

@article{MR0883539,
	author = {Eckmann, Jean-Pierre and Wittwer, Peter},
	doi = {10.1007/BF01013368},
	journal = {Journal of Statistical Physics},
	number = {3-4},
	pages = {455--475},
	title = {A complete proof of the {F}eigenbaum conjectures},
	volume = {46},
	year = {1987}}

@article{MR0727816,
	author = {Eckmann, Jean-Pierre and Koch, H. and Wittwer, Peter},
	doi = {10.1090/memo/0289},
	journal = {Memoirs of the American Mathematical Society},
	number = {289},
	pages = {vi+122},
	title = {A computer-assisted proof of universality for area-preserving maps},
	volume = {47},
	year = {1984}}

@phdthesis{Fejoz,
	author = {F\'{e}joz, Jacques},
	school = {Universit\'{e} Pierre et Marie Curie - Paris VI},
	title = {Periodic and quasi-periodic motions in the many-body problem},
	type = {Habilitation \`{a} diriger des recherches},
	year = {2010}}

@article{MR3990999,
	author = {G\'{o}mez-Serrano, Javier},
	doi = {10.1007/s40324-019-00186-x},
	journal = {SeMA Journal: Boletin de la Sociedad Espa\~{n}ola de Matem\'{a}tica Aplicada},
	number = {3},
	pages = {459--484},
	title = {Computer-assisted proofs in {PDE}: a survey},
	volume = {76},
	year = {2019}}

@article{MR3973675,
	author = {Guillot, Louis and Cochelin, Bruno and Vergez, Christophe},
	doi = {10.1002/nme.6049},
	journal = {International Journal for Numerical Methods in Engineering},
	number = {4},
	pages = {261--280},
	title = {A generic and efficient {T}aylor series-based continuation method using a quadratic recast of smooth nonlinear systems},
	volume = {119},
	year = {2019}}

@article{Gar-Ize-13,
	author = {Garc\'{i}a-Azpeitia, Carlos and Ize, J.},
	doi = {10.1016/j.jde.2012.08.022},
	journal = {Journal of Differential Equations},
	number = {5},
	pages = {2033--2075},
	title = {Global bifurcation of planar and spatial periodic solutions from the polygonal relative equilibria for the {$n$}-body problem},
	volume = {254},
	year = {2013}}

@article{Henon1974,
	author = {H{\'e}non, M.},
	doi = {10.1007/BF01586865},
	journal = {Celestial Mechanics},
	number = {3},
	pages = {375--388},
	title = {Families of periodic orbits in the three-body problem},
	volume = {10},
	year = {1974}}

@article{nGonPaper,
	author = {Hoppe, R.},
	journal = {Archives of Mathematical Physics},
	number = {218},
	title = {Erweiterung der bekannten Speciall\"{o}sung des Dreik\"{o}rperproblems},
	volume = {64},
	year = {1879}}

@article{Henot2021,
	author = {Olivier H{\'e}not},
	doi = {10.3934/jcd.2021013},
	issue = {3},
	journal = {Journal of Computational Dynamics},
	pages = {307},
	title = {On polynomial forms of nonlinear functional differential equations},
	volume = {8},
	year = {2021}}

@misc{Henot2021-2,
	author = {Olivier H\'{e}not},
	doi = {10.5281/zenodo.5705258},
	howpublished = {\url{https://github.com/OlivierHnt/RadiiPolynomial.jl}},
	note = {Software},
	title = {{R}adii{P}olynomial.jl},
	year = {2021}}

@misc{Henot2024,
	author = {Olivier H\'{e}not},
	howpublished = {\url{https://github.com/OlivierHnt/MarchalConjecture.jl}},
	note = {Software},
	title = {{M}archal{C}onjecture.jl},
	year = {2024}}

@article{Jorba2005,
	author = {Jorba, \`{A}ngel and Zou, Maorong},
	doi = {10.1080/10586458.2005.10128904},
	issue = {1},
	journal = {Experimental Mathematics},
	pages = {99--117},
	title = {A software package for the numerical integration of {ODE}s by means of high-order {T}aylor methods},
	volume = {14},
	year = {2005}}

@article{MR2012847,
	author = {Kapela, Tomasz and Zgliczy\'{n}ski, Piotr},
	doi = {10.1088/0951-7715/16/6/302},
	journal = {Nonlinearity},
	number = {6},
	pages = {1899--1918},
	title = {The existence of simple choreographies for the {$N$}-body problem -- a computer-assisted proof},
	volume = {16},
	year = {2003}}

@article{MR2312391,
	author = {Kapela, Tomasz and Sim\'{o}, Carles},
	doi = {10.1088/0951-7715/20/5/010},
	journal = {Nonlinearity},
	number = {5},
	pages = {1241--1255},
	title = {Computer assisted proofs for nonsymmetric planar choreographies and for stability of the {E}ight},
	volume = {20},
	year = {2007}}

@article{MR3622273,
	author = {Kapela, Tomasz and Sim\'{o}, Carles},
	doi = {10.1088/1361-6544/aa4ff3},
	journal = {Nonlinearity},
	number = {3},
	pages = {965--986},
	title = {Rigorous {KAM} results around arbitrary periodic orbits for {H}amiltonian systems},
	volume = {30},
	year = {2017}}

@article{MR1420838,
	author = {Koch, Hans and Schenkel, Alain and Wittwer, Peter},
	doi = {10.1137/S0036144595284180},
	journal = {SIAM Review},
	number = {4},
	pages = {565--604},
	title = {Computer-assisted proofs in analysis and programming in logic: a case study},
	volume = {38},
	year = {1996}}

@book{Knuth1981,
	author = {Knuth, Donald E.},
	edition = {Second},
	publisher = {Addison-Wesley Publishing Co., Reading, Mass.},
	title = {The Art of Computer Programming, Volume 2: Seminumerical Algorithms},
	year = {1981}}

@article{lagrangeTriangle,
	author = {Lagrange, Joseph-Louis},
	journal = {Oeuvres},
	pages = {229--331},
	title = {Essai sur le probl\`{e}me des trois corps},
	volume = {6},
	year = {1772}}

@article{MR0648529,
	author = {Lanford III, Oscar E.},
	doi = {10.1090/S0273-0979-1982-15008-X},
	journal = {American Mathematical Society Bulletin New Series},
	number = {3},
	pages = {427--434},
	title = {A computer-assisted proof of the {F}eigenbaum conjectures},
	volume = {6},
	year = {1982}}

@article{MR0759197,
	author = {Lanford III, Oscar E.},
	doi = {https://doi.org/10.1016/0378-4371(84)90262-0},
	journal = {Physica A: Statistical Mechanics and its Applications},
	number = {1},
	pages = {465-470},
	title = {Computer-assisted proofs in analysis},
	volume = {124},
	year = {1984}}

@article{Lessard2016,
	author = {Lessard, Jean-Philippe and Mireles James, Jason D. and Ransford, Julian},
	doi = {10.1016/j.physd.2016.02.007},
	journal = {Physica D: Nonlinear Phenomena},
	pages = {174-186},
	title = {Automatic differentiation for {F}ourier series and the radii polynomial approach},
	volume = {334},
	year = {2016}}

@book{MR1124619,
	author = {Marchal, Christian},
	publisher = {Elsevier Science Publishers, B. V., Amsterdam},
	series = {Studies in Astronautics},
	title = {The three-body problem},
	volume = {4},
	year = {1990}}

@article{P12family,
	author = {Marchal, Christian},
	journal = {Celestial Mechanics and Dynamical Astronomy},
	pages = {279--298},
	title = {The family {$P_{12}$} of the three-body problem -- the simplest family of periodic orbits, with twelve symmetries per period},
	volume = {78},
	year = {2000}}

@article{scholarpedia,
	author = {Montgomery, Richard},
	journal = {Scholarpedia},
	title = {{$N$}-body choreographies},
	volume = {\url{http://www.scholarpedia.org/article/N-body_choreographies}},
	year = {2010}}

@article{MR1610784,
	author = {Montgomery, Richard},
	doi = {10.1088/0951-7715/11/2/011},
	journal = {Nonlinearity},
	number = {2},
	pages = {363--376},
	title = {The {$N$}-body problem, the braid group, and action-minimizing periodic solutions},
	volume = {11},
	year = {1998}}

@article{DiscoveryEight,
	author = {Moore, Cristopher},
	doi = {10.1103/PhysRevLett.70.3675},
	issue = {24},
	journal = {Physical Review Letters},
	pages = {3675--3679},
	publisher = {American Physical Society},
	title = {Braids in classical dynamics},
	volume = {70},
	year = {1993}}

@book{MR3971222,
	author = {Nakao, Mitsuhiro and Plum, Michael and Watanabe, Yoshitaka},
	doi = {10.1007/978-981-13-7669-6},
	publisher = {Springer Singapore},
	title = {Numerical Verification Methods and Computer-Assisted Proofs for Partial Differential Equations},
	year = {2019}}

@incollection{MR1884902,
	author = {Sim\'{o}, Carles},
	booktitle = {Celestial Mechanics: Dedicated to Donald Saari for his 60th Birthday},
	doi = {10.1090/conm/292/04926},
	pages = {209--228},
	publisher = {American Mathematical Society, Providence, RI},
	series = {Contemporary Mathematics},
	title = {Dynamical properties of the figure eight solution of the three-body problem},
	volume = {292},
	year = {2002}}

@incollection{MR1905315,
	author = {Sim\'{o}, Carles},
	booktitle = {European Congress of Mathematics},
	pages = {101--115},
	publisher = {Birkh\"{a}user, Basel},
	series = {Progress in Mathematics},
	title = {New families of solutions in {$N$}-body problems},
	volume = {201},
	year = {2001}}

@book{MR2807595,
	author = {Tucker, Warwick},
	publisher = {Princeton University Press},
	title = {Validated Numerics: A Short Introduction to Rigorous Computations},
	year = {2011}}

\end{document}